\documentclass[11pt,letterpaper]{amsart}

\usepackage{dsfont}
\usepackage{amsmath,amssymb,amsthm,amsxtra,mathrsfs}
\usepackage{graphicx}
\usepackage{cancel,ulem}
\usepackage{color}
\usepackage[all,cmtip]{xy}
\usepackage{hyperref}
\usepackage{enumitem}
\usepackage{color}
\usepackage{bm}
\usepackage{subcaption}
\usepackage{tikz}
\usetikzlibrary{arrows}
\usepackage{accents}
\usepackage{bbm}

\newcommand{\jac}{\text{Jac}}

\def\r{{\rangle}}
\def\l{{\langle}}
\def\S{{\mathbb S}}

\def\R{{\mathbb R}}

\newtheorem{theorem}{Theorem}
\newtheorem{lemma}[theorem]{Lemma}
\newtheorem{corollary}[theorem]{Corollary}
\newtheorem{proposition}[theorem]{Proposition}

\theoremstyle{definition}
\newtheorem{definition}[theorem]{Definition}

\theoremstyle{remark}
\newtheorem{remark}[theorem]{Remark}

\numberwithin{equation}{section}
\numberwithin{theorem}{section}
\numberwithin{problem}{section}

\begin{document}

\begin{abstract}
In this paper we construct global dispersive solutions to the space inhomogeneous kinetic wave equation (KWE) which propagate $L^1_{xv}$ -- moments and conserve mass, momentum and energy. We prove that they scatter, and that the wave operators mapping the initial data to the scattering states are 1-1, onto and continuous in a suitable topology.

This is the first global existence result for strong solutions for KWE. This contrasts with prior global existence results for mild solutions, which satisfy a transported version of the equation but do not solve the equation itself. 

Our proof is carried out entirely in physical space and combines dispersive estimates for the free transport with new trilinear bounds for the gain and loss operators of the KWE on weighted Lebesgue spaces. The main difficulty is the fast growth of the hard-sphere kernel. Our fundamental tool to handle it is a novel collisional averaging estimate.

We also show that the nonlinear evolution preserves positivity forward in time. For this, we use the Kaniel-Shinbrot iteration scheme \cite{KS}, properly initialized to ensure successive approximations are dispersive. 
\end{abstract}

\title[Strong solutions to the Inhomogeneous Kinetic Wave equation]{Global existence of strong solutions to the Inhomogeneous Kinetic Wave Equation}

\author[I. Ampatzoglou]{Ioakeim Ampatzoglou}
\address{Baruch College \& The Graduate Center, City University of New York, Newman Vertical Campus, 55 Lexington Ave New York, NY, 10010, USA}
\email{ioakeim.ampatzoglou@baruch.cuny.edu}

\author[T. L\'eger ]{Tristan L\'eger}
\address{Courant Institute of Mathematical Sciences, 251 Mercer Street, New York, NY 10012, USA}
\email{tbl248@nyu.edu}




\maketitle

\tableofcontents

\section{Introduction}

\subsection{Background}
In this paper we consider the 3D space inhomogeneous kinetic wave equation
\begin{equation} \label{KWE}
\begin{cases}
\partial_t f + v \cdot \nabla_x f = \mathcal{C}[f]\\
f(t=0)=f_0
\end{cases}
\end{equation}
where $f:(-T,T) \times \mathbb{R}_x^3 \times \mathbb{R}_v^3 \rightarrow \mathbb{R}$, $T>0$ and $f_0:\R^3_x\times\R^3_v\to\R$.

The collision operator $\mathcal{C}$ is defined as 
\begin{align} \label{collision}
\begin{split}
    \mathcal{C}[f] & := \int_{\mathbb{R}^{9}} \delta (\Sigma) \, \delta(\Omega) \, f f_1 f_2 f_3 \, \big(\frac{1}{f} + \frac{1}{f_1} - \frac{1}{f_2} - \frac{1}{f_3} \big) \,dv_1 dv_2 dv_3  \\
    \Sigma&:=v+v_1 - v_2 - v_3 \\
    \Omega &:= \vert v \vert^2 + \vert v_1 \vert^2 - \vert v_2 \vert^2 - \vert v_3 \vert^2 .
\end{split}
\end{align}
This equation is well-known in the field of wave turbulence, which has seen a surge of interest in recent years. Although some basic concepts and principles had already been laid out   in 1929 by R. Peierls \cite{Pe} and independently in 1962 by K. Hasselman \cite{ha62,ha63}, a connection to hydrodynamic turbulence uncovered by V. Zakharov, V. L'vov and G. Falkovich \cite{ZLF} contributed to a dramatic uptick in popularity of the subject. 

This theory aims to describe the out-of-equilibrium dynamics of physical system composed of random waves. More precisely the objective is to track and predict statistical properties of solutions to these equations, such as transfers of energy between scales. This is typically achieved by deriving effective equations (of wave kinetic type) for the quantities of interest. We refer to the lecture notes of S. Nazarenko \cite{Na} for a comprehensive list of examples where this theory was successfully implemented.

At the mathematical level the theory has matured: there are now rigorous derivations of wave kinetic equations for several models. The most studied is the homogeneous kinetic wave equation with Laplacian dispersion relation (it corresponds to solutions to \eqref{KWE} that do not depend on $x$). It was shown to capture the statistical properties of a system of random waves modelled by the cubic nonlinear Schr\"{o}dinger equation on large time scales. The validity of the derivation was proved on increasingly large time intervals by various groups until the optimal time (so called kinetic time) was reached by Y. Deng and Z. Hani. We refer the reader to the works \cite{BGHS}, \cite{DH1}, \cite{DH2}, \cite{DH3}, \cite{DH4}, \cite{CG1}, \cite{CG2} for details. In the inhomogeneous setting, the first derivation result was obtained for the 3-wave kinetic wave equation and quadratic nonlinearities, albeit on a shorter time scale,  by the first author, C. Collot, and P. Germain \cite{AmCoGer}. Later, inhomogeneous turbulence was also derived from the Wick (NLS) by Z. Hani, J. Shatah and H. Zhu in \cite{HaShZh}. Other models have also been considered, such as KdV-type equations by G. Staffilani and M.-B. Tran \cite{ST}, as well as X. Ma \cite{Ma}. This general program is also presently being pursued by Y. Deng, A. Ionescu and F. Pusateri for water waves \cite{DIP}. 

With the derivation of these kinetic wave equations now on firm mathematical ground, it is legitimate to start analyzing them to gain further insight into the out-of-equilibrium behavior of systems of interacting waves. There are still relatively few results in this direction. In the case of the homogeneous kinetic wave equation for example, only local well-posedness is known for strong solutions by work of P. Germain, A. Ionescu and M.-B. Tran \cite{GeIoTr}. Global weak solutions were constructed, and their asymptotic behavior studied by M. Escobedo and J. Vel\'{a}zquez in \cite{EV}. Moreover, the stability of equilibria and Kolmogorov-Zakharov spectra was studied in \cite{me23,EsMe24, codige24}. For the space inhomogeneous equation \eqref{KWE}, we will give a more detailed account of prior results  below. 

Of course the kinetic community has already studied equations similar to \eqref{KWE}. In fact it corresponds to the cubic part of nonlinearity of the Boltzmann-Nordheim equation. This gives another motivation for studying the model \eqref{KWE}, since the analysis can be combined with that of the Boltzmann equation to obtain results for the quantum Boltzmann model.

\subsection{Results obtained}
We now focus more specifically on equation \eqref{KWE}. We first mention that global existence of \textit{mild} solutions has been proved by two different methods in \cite{Am} and \cite{AmMiPaTa24}. Mild solutions satisfy a transported version of \eqref{KWE}, but not the equation itself.

In the present article, we prove the first global existence result for \textit{strong} solutions to \eqref{KWE}. Note that our solutions are also mild solutions, while the converse is not true. At the technical level, this is due to the fact that mild solutions are constructed in spaces that are weighted in $x.$ Since free transport is not an isometry on such spaces, mild solutions do not correspond to solutions of \eqref{KWE}. 

To avoid the pitfall of $x$ weighted spaces, and the subsequent inevitable restriction to mild solutions, our proof relies on dispersive estimates for free transport instead of closeness to traveling maxwellians. The use of dispersive methods is not without precedent in the kinetic community. For the Boltzmann equation they were used to study local well-posedness at low regularity. In that direction, see works of T. Chen, R. Denlinger and N. Pavlovi\'{c} \cite{CDP-LWP,CDP-moments,CDP-GWP}, as well as X. Chen and J. Holmer \cite{CH}, X. Chen, S. Shen and Z. Zhang \cite{CSZ1,CSZ2}. All these papers actually focused on the Fourier transform of the solution rather than work in physical space. We also mention the work of D. Arsenio \cite{Arsenio} that relies on Strichartz estimates for the free transport to construct weak solutions to the Boltzmann equation for soft potentials. The problems mentioned above are subcritical with respect to scattering. A currently very active related line of research consists in proving global existence results near vacuum for equations that are critical with respect to scattering. Most closely related to our work, we mention \cite{Chaturvedi} which treats the Boltzmann equation with moderately soft potentials and no angular cutoff. The main difficulty we address in this article is different, and comes from the hard-sphere nature of the potential in \eqref{KWE}. Kinetic equations with such cross-sections typically have a much worse behavior than equations with soft potentials. For example it is easy to see that the equation with right-hand side equal to the loss term of the Boltzmann equation 
\begin{align*}
    \partial_t f + v \cdot \nabla_x f = f(v) \int_{\mathbb{R}^3} \vert v-v_1 \vert ^\gamma f(v_1) \, dv_1
\end{align*}
is locally ill-posed in any space of the form $\langle v \rangle^{-M} L^{\infty}_{xv}$ for large enough $M>0$ and any $\gamma >0.$ For comparison, recall that the hard-sphere potential present in \eqref{KWE} corresponds to $\gamma=1.$

Therefore the main contribution of the present article is the discovery and use of a new collisional averaging effect (see Lemma \ref{singint}). It overcomes the fast growth of the hard-sphere kernel in the context of the kinetic wave equation while not requiring localization in the spatial variable, i.e. weights in $x.$

To the best of our knowledge, the closest result to ours is the very recent article \cite{HJKL}, where global solutions to the Boltzmann equation with hard spheres are constructed by means of Strichartz estimates. Of course here we deal with a different equation. We also emphasize that our method to prove Lebesgue boundedness of the collision operator is also very different, relying on purely kinetic tools instead of Fourier analysis. Moreover we are able to prove a number of desirable physical properties for our solutions, namely that they propagate moments, conserve mass, energy and momentum and preserve positivity forward in time. Informally, our results can be stated as 
\begin{theorem}
The equation \eqref{KWE} with small and sufficiently decaying initial data $f_0$ has a unique global dispersive solution. Moreover as $t \to \pm \infty$ this solution scatters, that is there exist $f_{+\infty}, f_{-\infty}$ such that 
\begin{align*}
    f(t,x,v) - f_{+\infty}(x-vt,v) \longrightarrow_{t \to +\infty} 0, \quad
    f(t,x,v) - f_{-\infty} (x-vt,v) \longrightarrow_{t \to -\infty} 0,
\end{align*}
in a suitable sense.

Moreover these solutions propagate $L^1_{xv}$-- moments: if $\Vert \l v \r^N f(0) \Vert_{L^1_{xv}}$ is small enough, then $\Vert \l v \r^N f(t) \Vert_{L^1_{xv}}$ remains uniformly bounded in time, and they conserve total mass, momentum and energy i.e.
$$\partial_t \int_{\R^6}\phi f(t)\,dv\,dx=0,\quad \phi\in\{1,v,|v|^2\}.$$

Finally the nonlinear flow preserves positivity forward in time: if $f(0) \geq 0,$ then for all $t \ge 0$ we have $f(t) \geq 0.$
\end{theorem}

To refine the analysis, we describe the scattering states using concepts from scattering theory, namely wave operators. It is indeed natural to seek further information on the scattering states $f_{\pm \infty}$, about their general shape for example. By analogy with the Boltzmann equation, it could be expected that they have to be thermodynamic equilibria, called Rayleigh-Jeans distributions (the counterparts of maxwellians in kinetic wave theory). Indeed L. Desvillettes and C. Villani proved in \cite{DV} that the scattering states have to be maxwellians for the inhomogeneous Boltzmann equation set on a torus. We answer in the negative and prove that the scattering states can be any reasonably fast decaying function near vacuum. 

A scattering result of this kind has been obtained for \textit{mild} solutions to the Boltzmann equation by C. Bardos, I.M. Gamba, F. Golse and C.D. Levermore  in \cite{BGGL}. However as stressed previously, this is really a statement about solutions to a transported version of the Boltzmann equation, not the equation itself. Indeed the authors work in spaces exponentially weighted in $x$ and $v$, whereas we avoid weights in $x$ and work in spaces that are polynomially weighted in $v$ only. We emphasize that this is not a mere technicality and leads to conceptual differences with existing works in the literature: it allows us to construct strong solution (the standard notion of solutions in PDE theory) for the equation, rather than mild solutions as in \cite{BGGL,Am,AmMiPaTa24}. 

Moreover in the setting of wave turbulence polynomially weighted spaces are more physically relevant than exponentially weighted ones, especially when it comes to long-time dynamics. Indeed, thermodynamic equilibrium is reached for distributions that exhibit polynomial decay, unlike for the Boltzmann equation where such states are maxwellians.

Besides allowing us to describe the asymptotic behavior of solutions (namely scattering), another advantage of our method is that it straightforwardly yields global propagation of $L^1_{xv}$ -- moments for our solutions. This is far from evident a priori; in fact, to the best of our knowledge, for hard cutoff potentials, propagation of $L^1_{xv}$ -- moments is not known for the inhomogeneous Boltzmann equation \cite{Villani}, although moment production estimates are well studied in the spatially homogeneous case, see e.g. \cite{de93, bo97, miwe99, ACGM}.  

Our method of proof is versatile, and we expect that it can be implemented for many other kinetic models. In this direction, a similar result for the Boltzmann equation will be the subject of an upcoming work \cite{AmLeBoltz}.

\subsection{Ideas of the proof}
Our proof is perturbative in nature: we show that the nonlinearity is dominated by the free transport in the small data regime. Since this operator has well-known dispersive properties, this allows us to control the evolution globally in time.

To use a bootstrap-based perturbative approach modelled on similar arguments for dispersive equations (e.g. the nonlinear Schr\"{o}dinger equation), it is necessary to know mapping properties of the collision operator on (possibly weighted) Lebesgue spaces. However nonlinearities in kinetic equations are notoriously hard to work with due to their complicated expression as integrals on resonant manifolds. Dealing with this difficulty is a key step in our paper. In fact this specific problem has a long history for the Boltzmann equation, whose theory is much more advanced than that of the kinetic wave equation. A large literature is devoted solely to establishing boundedness of the collision operator on weighted Lebesgue spaces (see the papers of R. Alonso, E. Carneiro \cite{AC}, and the same authors with I.M. Gamba \cite{ACG}, as well as references therein). Using these estimates, distributional and classical solutions for the Boltzmann equation for soft potentials were constructed by R. Alonso and I.M. Gamba in \cite{AlGa09}. 

The fact that we work in polynomially weighted spaces adds a layer a complexity to the problem. Indeed the presence of linear growth in the kernel of the collision operator could lead to believe that estimates cannot be closed (naively the right-hand side of the equation always has one more $v$ weight than the linear part). To deal with this problem, we leverage the decay of a singular integral (see Lemma \ref{singint}) to offset the linear growth of the kernel. This can be seen as a kind of regularization by angular averaging. To the best of our knowledge this observation was not made in the kinetic literature previously. The same difficulty is present for the Boltzmann equation with hard potentials. Despite intense efforts to obtain sharp weighted Lebesgue estimates of the collision operator, this case was only very recently settled in \cite{HJKL}, where the authors build on the extensive literature mentioned above. The regularization mechanism we uncover for \eqref{KWE} is different (and simpler) than the one used in \cite{HJKL} and its precursors. 

Of course the bounds we obtain are much less favorable than those typically available in more standard dispersive problems, such as the nonlinear Schr\"{o}dinger equation. Therefore the bootstrap assumptions have to be carefully chosen to accommodate the rough nonlinearity. As a result multiple estimates quantifying the decay, amplitude and integrability of the solution have to be propagated to close the argument. 

As mentioned above, an upside of our method of proof is that a number of properties of the solution can be directly deduced from the analysis. For example we are able to show that the solutions propagate $L^1_{xv}$ -- moments. This in turn allows us to prove that the natural conservation laws (mass, momentum, energy) are satisfied. We also stress that our proof is entirely perturbative. This differs from techniques used for the Boltzmann equation, which consist in solving the gain only equation and using the Kaniel-Shinbrot scheme to control the full evolution. In particular, we can prove lower bounds for the moments, which cannot be done with the aforementioned Boltzmann methods. 

Finally we prove that the nonlinear flow preserves positivity forward in time. This is a non-perturbative argument which uses the gain-loss structure of the collision operator. Therefore one expects the equation to display time irreversibility, unlike the transport which is time reversible. This is reflected in the fact that our argument only applies to positive times, although the free transport propagates positivity both forward and backward in time. To show positivity of the solution for positive times, we use the Kaniel-Shinbrot \cite{KS} iteration scheme which is a classical tool in collisional kinetic theory, e.g. \cite{KS,illner-shinbrot,beto85,to86,AlGa09,amgapata22}; however it was not until recently exploited in the context of wave turbulence by the first author in \cite{Am}, and later by the first author and collaborators in \cite{AmMiPaTa24}. In our dispersive setting though, the use of this method is more delicate. Indeed it is harder to guarantee dispersion of the iterates. To remedy this, we carefully initialize the scheme by solving the gain only equation \eqref{KWE} and using its dispersive bounds. We note that similar ideas of initializing the scheme by solving the gain equation were first used in the works \cite{illner-shinbrot,CDP-GWP,CSZ2}.

\subsection*{Acknowledgements}
I.A. was supported by the NSF grant No. DMS-2418020. T.L. was supported by the Simons grant on wave turbulence. The authors thank Jalal Shatah for his comments on an earlier version of the manuscript. T.L. thanks Sanchit Chaturvedi for pointing out several relevant references.

\subsection{The kinetic model}
The collision operator can be  split in gain-loss form
\begin{equation}
\mathcal{C}[f]=\mathcal{G}[f]-\mathcal{L}[f],    
\end{equation}
where the gain and loss operators are respectively given by  
\begin{align}
\mathcal{G}[f]&=\int_{\R^9}\delta(\Sigma)\delta(\Omega)f_2f_3(f+f_1)\,dv_1\,dv_2\,dv_3,\label{gain operator}\\
\mathcal{L}[f]&=\int_{\R^9} \delta(\Sigma)\delta(\Omega) ff_1(f_2+f_3)\,dv_1\,dv_2\,dv_3\label{loss operator},
\end{align}
the resonant manifolds by
\begin{align}
   \Sigma&=v+v_1-v_2-v_3, \label{Sigma}\\
\Omega&=|v|^2+|v_1|^2-|v_2|^2-|v_3|^2, \label{Omega}
\end{align}
and we use the standard notation $f:=f(x,v)$, $f_i=f(x,v_i)$, $i=1,2,3$. Notice that both $\mathcal{G},\mathcal{L}$ are positive operators and monotone for non-negative inputs.

Relabeling the variables, one can see that the collisional operator $\mathcal{C}[f]$ satisfies the following weak formulation 
\begin{equation}\label{weak formulation}
\int_{\R^3}\mathcal{C}[f]\phi\,dv= \frac{1}{2}\int_{\R^{12}}\delta(\Sigma)\delta(\Omega)f_1f_2f_3\left(\phi+\phi_1-\phi_2-\phi_3\right)\,dv_1\,dv_2\,dv_3\,dv,
\end{equation}
where $\phi$ is a test function appropriate for all the above integrations to make sense. Choosing $\phi\in\{1,v,|v|^2\}$ and using the resonant conditions, one can formally see that a solution $f$ to \eqref{KWE}  conserves the total mass, momentum and energy i.e.
\begin{equation}\label{conservation laws formal} \partial_t\int_{\R^6}f\phi\,dv\,dx=0,\quad \phi\in\{1,v,|v|^2\}.
\end{equation}

Now we focus on two ways of equivalently expressing the collisional operator that will be useful throughout the paper.  

\subsubsection{Multilinear decomposition} \label{subsection:multilin}
Instead of further decomposing the gain and loss operators into multilinear terms in the obvious way, we will decompose them taking advantage of the symmetries of the equation. In our context, that  simple observation is of crucial importance in controlling the nonlinear effects caused by the resonances.

More specifically, let us introduce the sharp cutoff $b=\mathds{1}_{(0,+\infty)}$. Then, we can write
\begin{align*}
\mathcal{G}[f]&=\int_{\R^9}\delta(\Sigma)\delta(\Omega)f_2f_3(f+f_1)b\left(\left(v-v_1\right)\cdot\left(v_2-v_3\right)\right)\,dv_1\,dv_2\,dv_3\\
&\hspace{1cm}+\int_{\R^9}\delta(\Sigma)\delta(\Omega)f_2f_3(f+f_1)b\left(\left(v_1-v\right)\cdot\left(v_2-v_3\right)\right)\,dv_1\,dv_2\,dv_3\\
&=2\int_{\R^9}\delta(\Sigma)\delta(\Omega) f_2f_3(f+f_1)b\left(\left(v-v_1\right)\cdot\left(v_2-v_3\right)\right)\,dv_1\,dv_2\,dv_3,
\end{align*}
where to obtain the last line we exchanged $v_2$ with $v_3$ in the second integrand and used symmetry of the product $f_2f_3$.

Hence, we can write  
\begin{equation}\label{gain decomposed}
    \mathcal{G}[f]=\mathcal{G}_1[f,f,f]+\mathcal{G}_2[f,f,f],
\end{equation}
where
we use the multilinear notations
\begin{align}
\mathcal{G}_1[f,g,h]&:=  2\int_{\R^9}\delta(\Sigma)\delta(\Omega)f_1g_2h_3 \, b\left(\left(v-v_1\right)\cdot\left(v_2-v_3\right)\right)\,dv_1\,dv_2\,dv_3\label{G_1},  \\
\mathcal{G}_2[f,g,h]&:=  2\int_{\R^9}\delta(\Sigma)\delta(\Omega)fg_2h_3 \, b\left(\left(v-v_1\right)\cdot\left(v_2-v_3\right)\right)\,dv_1\,dv_2\,dv_3\label{G_2}.
\end{align}
With an identical argument, using the symmetry of the sum $f_2+f_3$ instead, we can write
\begin{equation}\label{loss decomposed}
   \mathcal{L}[f]=\mathcal{L}_1[f,f,f]+\mathcal{L}_2[f,f,f], 
\end{equation}
where
\begin{align}
\mathcal{L}_1[f,g,h]&:=  2\int_{\R^9}\delta(\Sigma)\delta(\Omega)fg_1h_3 \, b\left(\left(v-v_1\right)\cdot\left(v_2-v_3\right)\right)\,dv_1\,dv_2\,dv_3\label{L_1},  \\
\mathcal{L}_2[f,g,h]&:=  2\int_{\R^9}\delta(\Sigma)\delta(\Omega)fg_1h_2 \, b\left(\left(v-v_1\right)\cdot\left(v_2-v_3\right)\right)\,dv_1\,dv_2\,dv_3\label{L_2}.
\end{align}
It is evident that the operators $\mathcal{G}_1,\mathcal{G}_2,\mathcal{L}_1,\mathcal{L}_2$ are multilinear, positive, and increasing for non-negative inputs.

Under this notation, the collisional operator can be equivalently written as
\begin{equation}\label{collisional operator multilinear}
\begin{aligned}
  \mathcal{C}[f] &=   \mathcal{G}_1[f,f,f] + \mathcal{G}_2[f,f,f] - \mathcal{L}_1[f,f,f] - \mathcal{L}_2[f,f,f].
\end{aligned}
\end{equation}
Expression \eqref{collisional operator multilinear} will be useful for estimating the impact of the resonances in the equation.
\subsubsection{Collision frequency}
Another important aspect for our context is that the loss operator  is local in $v$, i.e. we have
\begin{equation}\label{locality of loss}
\mathcal{L}[f]=f\mathcal{R}[f,f],
\end{equation}
where  
\begin{align}\label{collision frequency}
\mathcal{R}[g,h]:=\int_{\R^9} \delta(\Sigma)\delta(\Omega) g_1(h_2+h_3)\,dv_1\,dv_2\,dv_3,
\end{align}
can be interpreted as the collision frequency. We note that $\mathcal{R}$ is also a positive operator and increasing for non-negative inputs, however it is not linear in the second argument. 

With this notation, the collisional operator $\mathcal{C}[f]$ can be equivalently written as
\begin{equation}\label{collisional operator frequency}
\mathcal{C}[f]=\mathcal{G}[f]-f\mathcal{R}[f,f]    
\end{equation}
Expression \eqref{collision frequency} will be useful for exploiting the monotonicity properties of the equation.

Finally we record one last identity, proved exactly like \eqref{gain decomposed}
\begin{align} \label{collision equal loss}
    f \mathcal{R}[g,g] = \mathcal{L}_1[f,g,g] + \mathcal{L}_2[f,g,g].
\end{align}
\color{black}

\subsection{Parametrization of the resonant manifolds}
In this subsection we parametrize the resonant manifolds, and convert the operators to quantum Boltzmann type operators. This is possible due to the Laplacian dispersion relation. In other words, we can view the interaction of $4$ waves with Laplacian dispersion relation in $d=3$ as the collision between two incoming hard spheres. With this parametrization in hand, we will be able to address the weighted $L^p$ boundness of the collisional operators.

Using Lemma A.7. from \cite{AmMiPaTa24}, we may write
\begin{align}
\mathcal{G}_1[f,g,h]&=\frac{1}{4}\int_{\R^3\times\S^2}|v-v_1|f(v_1)g(v^*)h(v^*_1)b\left(\left(v-v_1\right)\cdot\sigma\right)\,d\sigma\,dv_1    \label{G_1 parametrized}\\
\mathcal{G}_2[f,g,h]&=\frac{1}{4}f(v)\int_{\R^3\times\S^2}|v-v_1|g(v^*)h(v^*_1)b\left(\left(v-v_1\right)\cdot\sigma\right)\,d\sigma\,dv_1    \label{G_2 parametrized}\\
\mathcal{L}_1[f,g,h]&=\frac{1}{4}f(v)\int_{\R^3\times\S^2}|v-v_1|g(v_1)h(v^*_1)b\left(\left(v-v_1\right)\cdot\sigma\right)\,d\sigma\,dv_1    \label{L_2 parametrized}\\
\mathcal{L}_2[f,g,h]&=\frac{1}{4}f(v)\int_{\R^3\times\S^2}|v-v_1|g(v_1)h(v^*)b\left(\left(v-v_1\right)\cdot\sigma\right)\,d\sigma\,dv_1    \label{L_1 parametrized},
\end{align}
where given $\sigma\in \S^2$ we denote
\begin{align}
 v^*&=\frac{v+v_1}{2}+\frac{|v-v_1|}{2}\sigma\label{v^*},\\
 v_1^*&=\frac{v+v_1}{2}-\frac{|v-v_1|}{2}\sigma\label{v_1^*}.
\end{align}
One can easily verify that in accordance with the resonant conditions,  momentum and energy is conserved i.e. given $\sigma\in \S^2$, we have
\begin{align}
v^*+v_1^*&=v+v_1,\label{conservation of momentum}\\
|v^*|^2+|v_1^*|^2&=|v|^2+|v_1|^2 \label{conservation of energy}.
\end{align}
Moreover, there hold the relations
\begin{align}
|v^*-v_1^*|&=|v-v_1|,\label{conservation of relative velocities}\\
(v^*-v_1^*)\cdot\sigma&=-(v-v_1)\cdot\sigma.\label{microreversibilty}
\end{align}
Finally, given $\sigma\in \S^2$ the transformation $T_\sigma:(v,v_1)\mapsto (v^*,v_1^*)$ is a linear involution of $\R^6$.

\subsection{Organization of the paper}
We start by proving trilinear bounds for the collision operator in Section \ref{section:trilinear}. This enables us to run a bootstrap-based argument to prove global well-posedness of \eqref{KWE} in Section \ref{section:GWP}. As a byproduct of the proof we obtain that the solution satisfies dispersive estimates. It also shows propagation of $L^1_{xv}$ -- moments, from which we also deduce that solutions satisfy the natural conservation laws of mass, momentum and energy. We then study the asymptotic behavior of solutions, and show that they scatter in Section \ref{section:scattering}. Moreover we establish various properties of the corresponding wave operators to describe which states can be reached by the solution. Finally we prove that the nonlinear evolution preserves positivity in Section \ref{sec:non-negative solution}. We stress that this is a non-perturbative result which requires to exploit monotonicity properties of the collision operator. 

\subsection{Notations}
When writing $A \lesssim B,$ we mean that there exists a numerical constant $C > 0$ such that $A \leq C B.$ 

We write $A \simeq B$ to signify that both $A \lesssim B$ and $B \lesssim A$ hold.

We use the standard japanese bracket notation: $\langle v \rangle : = \sqrt{1 + \vert v \vert^2},$ where $\vert \cdot \vert$ denotes the $\ell^2$ norm of the vector $v \in \mathbb{R}^d.$ 

For $v \in \mathbb{R}^3,$  we denote $\hat{v} : = v/\vert v \vert.$

We use Lebesgue spaces with mixed-norms, which we denote (for $1 \leq p,q \leq +\infty$):
\begin{align*}
    \Vert f \Vert_{L^p_x L^q_v} := \bigg( \int_{\mathbb{R}^3} \Vert f(x,\cdot) \Vert_{L^{q}_v}^p \, dx \bigg)^{\frac{1}{p}} .
\end{align*}
We will also use the shorthand $\Vert f \Vert_{L^p_{xv}} := \Vert f \Vert_{L^p_{x} L^p_v}.$

\section{Weighted $L^p$ boundedness of the collision operator} \label{section:trilinear}
In this section we prove trilinear estimates for the operators that make up the nonlinearity. 

We start by proving technical lemmas, chief among which a crucial angular averaging estimate, see Lemma \ref{singint} in the next section. These facts are then used to prove boundedness of the nonlinear terms on weighted Lebesgue spaces in the next sections. 

\subsection{Preliminary results}
We begin with the angular averaging lemma. This is key in controlling the linear growth of the hard-sphere kernel.

 \begin{lemma} \label{singint}
 Let $\gamma\in[0;2]$ and
\begin{align*}
F(v,v_1) : = \int_{\S^{2}}\frac{|v-v_1|^{ \gamma}}{\l v^*\r^{3}}\,d\sigma,\quad F_1(v,v_1) : = \int_{\S^{2}}\frac{|v-v_1|^{ \gamma }}{\l v_1^*\r^{3}}\,d\sigma ,
\end{align*}
where 
\begin{align*}
  v^*&=\frac{v+v_1}{2}+\frac{|v-v_1|}{2}\sigma,\\
  v_1^*&=\frac{v+v_1}{2}-\frac{|v-v_1|}{2}\sigma .
\end{align*}
Then, we have
\begin{align}\label{L inf bound}
\|F(v,v_1)\|_{L^\infty_{v,v_1}},\,\|F_1(v,v_1)\|_{L^\infty_{v,v_1}}\lesssim 1.
\end{align}

\end{lemma}

\begin{proof}

    Fix $v,v_1\in\R^3$. 
Substituting $\sigma\mapsto -\sigma$ we see that $F_1(v,v_1)=F(v,v_1)$, so it suffices to prove the claim for $F(v,v_1)$.

Let us write
$$v^*=V+\frac{|u|}{2}\sigma,$$
where $V=\frac{v+v_1}{2}$ and $u=v-v_1$. Let us also denote $E=|v|^2+|v_1|^2$.
We have $$\l v^*\r^2=1+|V|^2+\frac{|u|^2}{4}+|u||V|(\hat{V}\cdot\sigma).$$
Moreover, 

$$|V|^2+\frac{|u|^2}{4}=\frac{1}{4}(|v+v_1|^2+|v-v_1|^2)=\frac{|v|^2+|v_1|^2}{2}=\frac{E}{2},$$
thus $$\l v^*\r^2= 1+\frac{E}{2}+|u||V|(\hat{V}\cdot\sigma).$$

Let us compute
$$I(v,v_1)=\int_{\S^{2}}\frac{1}{\l v^*\r^3}\,d\sigma .$$
Applying the spherical coordinates integration formula
$$\int_{\S^2}b(\hat{V}\cdot\sigma)\,d\sigma=2\pi \int_{0}^\pi b(\cos\theta)\sin\theta\,d\theta=2\pi\int_{-1}^1 b(z)\,dz$$
for
$$b(z)=\left(1+\frac{E}{2}+|u||V|z\right)^{-3/2},$$
which is integrable in $[-1,1]$ since
$$|u||V|=2\frac{|u|}{2}|V|\leq \frac{|u|^2}{4}+|V|^2=\frac{E}{2},$$
we obtain
\begin{align*}
    I(v,v_1)&\simeq\frac{1}{|u||V|}\left(\left(1+\frac{E}{2}-|u||V|\right)^{-1/2}-\left(1+\frac{E}{2}+|u||V|\right)^{-1/2}\right)\\
    &=\frac{1}{|u||V|}\frac{\left(1+\frac{E}{2}+|u||V|\right)^{1/2}-\left(1+\frac{E}{2}-|u||V|\right)^{1/2}}{\left(1+\frac{E}{2}-|u||V|\right)^{1/2}\left(1+\frac{E}{2}+|u||V|\right)^{1/2}}\\
    &=\frac{1}{\left(1+\frac{E}{2}-|u||V|\right)^{1/2}\left(1+\frac{E}{2}+|u||V|\right)^{1/2}\left(\left(1+\frac{E}{2}-|u||V|\right)^{1/2}+\left(1+\frac{E}{2}+|u||V|\right)^{1/2}\right)}\\
    &\lesssim \frac{1}{1+E},
  \end{align*}
using again $|u||V|\leq E/2$.

Now the above inequality, triangle inequality, and the fact that $\gamma\in [0,2]$ yield
$$F(v,v_1)=|v-v_1|^\gamma I(v,v_1)\lesssim \frac{|v|^\gamma+|v_1|^\gamma}{1+|v|^2+|v_1|^2}\leq \l v\r^{\gamma-2}+\l v_1\r^{\gamma-2} \lesssim 1,$$
and \eqref{L inf bound} follows.
The proof is complete.
\end{proof}

We also record some better known results, starting with the following which was used in \cite{Arsenio}. For convenience of the reader we provide a proof.

\begin{lemma} \label{chgvarkin}
For $\sigma\in\S^{2}$, we define the map 
\begin{align}\label{transition map}
   \nu= R_{\sigma}(u) :=  \frac{u}{2} +  \frac{\vert u \vert}{2} \sigma.
\end{align}
Then $R_\sigma$ is a diffeomorphism from $\lbrace u\in\R^3: u \cdot \sigma \neq -  \vert u \vert \rbrace$ onto $\lbrace \nu\in\R^3:  \nu \cdot \sigma > 0 \rbrace$ with inverse
\begin{align}\label{inverse function}
    u=\big(R_{\sigma}\big)^{-1}(\nu) = 2  \nu  - \frac{\vert \nu \vert^2}{\sigma \cdot \nu} \sigma,
\end{align}
and Jacobian 
\begin{equation}\label{Jacobian}
    \jac\,(R_\sigma)^{-1}(\nu)=\frac{4 \vert \nu \vert^2}{(\sigma \cdot \nu)^2}.
\end{equation}
Also note that 
\begin{align} 
|u|&=\frac{|R_{\sigma}(u)|^2}{(R_{\sigma}(u)\cdot\sigma)}\label{magnitude},\\
\hat{u} \cdot \sigma &=  \bigg(\frac{2 \big(R_{\sigma}(u)\cdot \sigma\big)^2}{\vert R_{\sigma}(u) \vert^2} -1\bigg) . \label{angle}
\end{align}
\end{lemma}
\begin{proof}

The fact that $R_\sigma$ is a diffeomorphism from $\lbrace u\in\R^3: u \cdot \sigma \neq -  \vert u \vert \rbrace$ onto $\lbrace \nu\in\R^3:  \nu \cdot \sigma > 0 \rbrace$ with inverse given by \eqref{inverse function} follows immediately by substitution.

In order to compute the Jacobian, we differentiate \eqref{inverse function} in $\nu$, to obtain
$$D (R_{\sigma})^{-1}(\nu)=2 I_3-\sigma\nabla^Tf(\nu), $$
where $f(\nu)=\frac{|\nu|^2}{(\nu\cdot\sigma)}$. Thus, we obtain\footnote{we use the Linear Algebra identity $\det(\lambda I_n+vw^T)=\lambda^n(1+\lambda^{-1}v\cdot w)$, where $v,w\in\R^n$, $n\in\mathbb{N}$ and $\lambda\neq 0$.}
\begin{equation}\label{Jacobian proof}
    \jac (R_\sigma)^{-1}(\nu)=\det\left(2 I_3-\sigma\nabla^T f(\nu)\right)=8\left(1-\frac{1}{2}\nabla f(\nu)\cdot\sigma\right).
\end{equation}
We readily compute
$$\nabla f(\nu)\cdot\sigma=\frac{2(\nu\cdot\sigma)^2-|\nu|^2}{(\nu\cdot\sigma)^2}=2-\frac{|\nu|^2}{(\nu\cdot\sigma)^2}$$
which, combined with \eqref{Jacobian proof}, gives \eqref{Jacobian}.

Finally, expression \eqref{magnitude} follows directly by \eqref{transition map}, and expression \eqref{angle} follows by combining \eqref{magnitude} and \eqref{transition map}.  
\end{proof}

We now deduce integrability estimates for post-collision quantities from this result. We stress that the presence of the cut-off $b$, introduced in our manipulations from Subsection \ref{subsection:multilin}, is necessary to avoid the singularity in the change of variables of the previous Lemma \ref{chgvarkin}. 

\begin{lemma}\label{pre average lemma} The following estimates hold
\begin{align}
\sup_{v\in\R^3}\int_{\R^3\times \S^2}\,|\psi(v_1^*)|b\left(\frac{v-v_1}{|v-v_1|}\cdot\sigma\right)\,dv_1\,d \sigma&\lesssim \|\psi\|_{L^{1}}\label{estimate for v_1 pre},\\
\sup_{v_1\in\R^3}\int_{\R^3\times \S^2}\,|\psi(v^*)|b\left(\frac{v-v_1}{|v-v_1|}\cdot\sigma\right)\,dv\,d \sigma&\lesssim \|\psi\|_{L^{1}},\label{estimate for v pre}
\end{align}
where we denote  $v^*=\frac{v+v_1}{2}+\frac{|v-v_1|}{2}\sigma$, $v_1^*=\frac{v+v_1}{2}-\frac{|v-v_1|}{2}\sigma$. 
\end{lemma}

\begin{proof}
We first prove \eqref{estimate for v_1 pre}. Let 
$$I(v)=\int_{\R^3\times \S^2}\,|\psi(v_1^*)|b\left(\frac{v-v_1}{|v-v_1|}\cdot\sigma\right)\,dv_1\,d \sigma.$$
Notice that we can write $v_1^*=v-\frac{v-v_1}{2}-\frac{|v-v_1|}{2}\sigma$.
Setting $u=v-v_1$, and using Lemma \ref{chgvarkin},  we obtain
\begin{align*}
I(v)&=  \int_{\R^3\times\S^2}\left|\psi\left(v-\frac{u}{2}-\frac{|u|}{2}\sigma\right)\right|\,b(\hat{u}\cdot\sigma)\,du\,d\sigma \\
&=\int_{\S^2}\int_{\R^3}|\psi(v-R_\sigma(u))\,|b\left(\frac{2 \big(R_{\sigma}(u)\cdot \sigma\big)^2}{\vert R_{\sigma}(u) \vert^2} -1\right)\,du\,d\sigma\\
&\simeq \int_{\R^3}|\psi(v-\nu)|\int_{\S^2}\frac{b\left(2(\hat{\nu}\cdot\sigma)^2-1\right)}{(\hat{\nu}\cdot\sigma)^{2}}\,d\sigma\,d\nu\\
&\simeq  \int_{\R^3}|\psi(v-\nu)|\int_0^\pi \frac{b(\cos 2\theta)}{\cos^{2}\theta}\sin\theta\,d\theta\,d\nu\\
&=\int_{\R^3}|\psi(v-\nu)|\int_{[0,\pi/4]\cup[3\pi/4,\pi]}\frac{\sin\theta}{\cos^{2}\theta}\,d\theta\,d\nu\\
&\simeq \|\psi\|_{L^1}.
\end{align*}

The proof for \eqref{estimate for v pre} follows similarly by  writing $v^*=v_1+\frac{v-v_1}{2}+\frac{|v-v_1|}{2}\sigma$. 
\end{proof}

\subsection{The loss operators}
We are now in a position to prove boundedness results for the collision operator on weighted Lebesgue spaces. We start with loss operators.  

\begin{lemma}
We have, for $\mathcal{T} \in \lbrace \mathcal{L}_1, \mathcal{L}_2 \rbrace$ and $p \geq 1$ and $l \geq 0$

\begin{align} \label{trilossmain}
    \Vert \langle v \rangle^l \mathcal{T}[f,g,h] \Vert_{L^{p}_v} \lesssim \Vert \langle v \rangle^l f \Vert_{L^p_v} \Vert g \Vert_{L^1_v} \Vert \langle v \rangle^3 h \Vert_{L^{\infty}_v}.
\end{align}
Moreover 
\begin{align} \label{triloss1}
    \Vert \mathcal{L}_1[f,g,h] \Vert_{L^1_v} \lesssim \Vert \langle v \rangle f \Vert_{L^1_v} \Vert \langle v \rangle g \Vert_{L^{\infty}_v} \Vert h \Vert_{L^1_v},
\end{align}
and 
\begin{align} \label{triloss2}
    \Vert \mathcal{L}_2[f,g,h] \Vert_{L^1_v} \lesssim \Vert \langle v \rangle f \Vert_{L^{\infty}_v} \Vert \langle v \rangle g \Vert_{L^1_v} \Vert h \Vert_{L^1_v}.
\end{align}

\end{lemma}
\begin{proof}
For $p \geq 1$, we have
\begin{align*}
    \Vert \langle v \rangle^l \mathcal{L}_1[f,g,h] \Vert_{L^{p}_v} & \lesssim \Vert \langle v \rangle^l f \Vert_{L^p_v} \Bigg \Vert \int_{\mathbb{R}^3 \times \mathbb{S}^{2}} g(v_1) \vert v-v_1 \vert h(v_1^{*}) \,  d\sigma dv_1  \Bigg \Vert_{L^{\infty}_v}  \\
    & \lesssim \Vert \langle v \rangle^l f \Vert_{L^p_v} \Vert g \Vert_{L^1_v} \Bigg \Vert \int_{\S^{2}} \langle v_1^{*} \rangle^3 h(v_1^{*}) \frac{\vert v - v_1 \vert }{\langle v_1^{*} \rangle^3}  d\sigma \Bigg \Vert_{L^{\infty}_{v,v_1}}.
\end{align*}
The result follows by Lemma \ref{singint}. The proof for $\mathcal{L}_2$ is identical, substituting $v_1^{*}$ by $v^{*}.$ 
\newline 

 We now focus on proving \eqref{triloss1}. Let us denote $\widetilde{f}(v)=\l v\r f(v)$, $\widetilde{g}(v)=\l v\r g(v)$.  We have
\begin{align*}
\Vert \mathcal{L}_1[f,g,h](v) \Vert_{L^{1}_v} & \lesssim  \Vert \widetilde{g}\Vert_{L^\infty_v}  \int_{\mathbb{R}^6 \times \S^{2}}  \frac{\vert v-v_1 \vert}{\langle v_1 \rangle \langle v \rangle} |\widetilde{f}(v)|\,|h(v_1^{*})|\,b\big(\frac{v-v_1}{\vert v- v_1 \vert} \cdot \sigma \big) \, dv_1 d\sigma \,dv  \\
& \leq  \Vert \widetilde{g}\Vert_{L^\infty_v}  \int_{\mathbb{R}^6 \times \S^{2}} |\widetilde{f}(v)| \,|h(v_1^{*})|\,b\big(\frac{v-v_1}{\vert v- v_1 \vert} \cdot \sigma \big) \, dv_1 d\sigma \,dv\\
&\leq \|\widetilde{g}\|_{L^\infty_v}\|\widetilde{f}\|_{L^1}\left\|  \int_{\mathbb{R}^6 \times \S^{2}}  \,|h(v_1^{*})|\,b\big(\frac{v-v_1}{\vert v- v_1 \vert} \cdot \sigma \big) \, dv_1 d\sigma \,dv\right\|_{L^\infty_v},\\
&\leq \|\widetilde{f}\|_{L^1_v}\|\widetilde{g}\|_{L^\infty_v}\|h\|_{L^1_v},
\end{align*}
where to obtain the second line we used the elementary fact
\begin{align*}
   \frac{|v-v_1|}{\l v \r \l v_1\r}&\leq \frac{|v|+|v_1|}{\l v\r \l v_1\r} \leq \frac{1}{\langle v_1 \rangle} + \frac{1}{\langle v \rangle} \leq 2,
\end{align*}
and to obtain the last line we used \eqref{estimate for v_1 pre} from Lemma \ref{pre average lemma}. 

The proof of \eqref{triloss2} is similar, pulling $\|\widetilde{f}\|_{L^\infty_v}$ out first and then using \eqref{estimate for v pre}.

\end{proof}

\subsection{The gain operators}
In this subsection we prove boundedness results on Lebesgue spaces for the gain operators. The proofs are different from that of the previous subsection. In particular for the loss operators our proofs hold for super-hard potentials $\vert v-v_1 \vert^\gamma, \gamma \in (1;2]$. However this is not the case for the gain operators, see the proof of \eqref{G_1 L2 estimate} below for example. 

\begin{lemma}
For $p\ge 1$ and $l\geq 0$, there holds the estimate
\begin{equation}\label{G_2 weighted estimate}
\|\l v\r^l \mathcal{G}_2[f,g,h]\|_{L^p_v}\leq \|\l v\r^l f\|_{L^p_v}\|\l v\r g\|_{L^\infty_v}\|\l v\r h\|_{L^1_v}    
\end{equation}
\end{lemma}
\begin{proof}
Let us denote $\widetilde{g}(v)=\l v\r g(v)$, $\widetilde{h}(v)=\l v\r h(v)$. We take
\begin{align*}
    \|\l v\r^l\mathcal{G}_2[f,g,h]\|_{L^p_v}&\lesssim\|\l v\r^l f\|_{L^p_v}\left\|\int_{\R^3\times\S^2}\frac{|v-v_1|}{\l v^*\r\l v_1^*\r} \vert \widetilde{g}(v^*) \vert \vert \widetilde{h}(v_1^*) \vert b\left(\frac{v-v_1}{|v-v_1|}\cdot\sigma\right)\,dv_1\,d\sigma\right\|_{L^\infty_v}\\
    & \leq \|\l v\r^l f\|_{L^p_v} \|\widetilde{g}\|_{L^\infty_v} \left\|\int_{\R^3\times\S^2}\frac{|v-v_1|}{\l v^*\r\l v_1^*\r}\vert \widetilde{h}(v_1^*) \vert b\left(\frac{v-v_1}{|v-v_1|}\cdot\sigma\right)\,dv_1\,d\sigma\right\|_{L^\infty_v}\\
    &\lesssim \|\l v\r^l f\|_{L^p_v} \|\widetilde{g}\|_{L^\infty_v} \left\|\int_{\R^3\times\S^2}\vert \widetilde{h}(v_1^*) \vert b\left(\frac{v-v_1}{|v-v_1|}\cdot\sigma\right)\,dv_1\,d\sigma\right\|_{L^\infty_v}\\
    &\leq \|\l v\r^l f\|_{L^p_v} \|\widetilde{g}\|_{L^\infty_v}\|\widetilde{h}\|_{L^1_v},
\end{align*}
where in the last two lines we used \eqref{conservation of relative velocities} to obtain
$$\frac{|v-v_1|}{\l v^*\r \l v_1^*\r}=\frac{|v^*-v_1^*|}{\l v^*\r \l v_1^*\r} \leq 2,$$
and \eqref{estimate for v_1 pre} from Lemma \ref{pre average lemma}.
\end{proof}

\begin{lemma}
For $l \geq 0$, there holds the estimate
\begin{align}\label{G1 weighted estimate}
    \Vert \langle v \rangle^l \mathcal{G}_1[f,g,h] \Vert_{L^{\infty}_v}\lesssim\Vert f \Vert_{L^1_v} \Vert \langle v \rangle^l g \Vert_{L^{\infty}_v} \Vert \langle v \rangle^3 h \Vert_{L^{\infty}_v}+   \|\l v\r^l f\|_{L^\infty_v}\|\l v\r g\|_{L^\infty_v}\|\l v\r h\|_{L^1_v}
\end{align}

\end{lemma}
\begin{proof}
We split the estimate into two different cases based on the magnitude of $|v|$ and $|v_1|$. 

Fix $\delta>0$, 
and assume $|v_1|\leq\delta|v|$ and $(v-v_1)\cdot\sigma>0$. 
Writing $V=v+v_1$, $u=v-v_1$, triangle inequality implies 
$|u| \geq (1-\delta)|v|$ and $|V|\geq (1-\delta)|v|$. Moreover $V\cdot\sigma=(u+2v_1)\cdot\sigma>2v_1\cdot\sigma\geq -2|v_1|\geq -2\delta|v|$. Now
\begin{align*}
4|v^*|^2&=\Big|V+|u|\sigma\Big|^2=|V|^2+|u|^2+2|u|V\cdot\sigma\\
&\geq 2(1-\delta)^2|v|^2-2\delta(1-\delta)|v|^2\\
&= 2(1-\delta)(1-2\delta)|v|^2
\end{align*}
Choosing $\delta=1/4$ we obtain $|v^*|\geq \frac{\sqrt{3}}{4}|v|$. 
Hence, we have
\begin{equation}\label{comparing characteristics}
\mathds{1}_{|v_1|\leq |v|/4}b\left(\frac{v-v_1}{|v-v_1|}\cdot\sigma\right)\leq\mathds{1}_{|v^*|\geq \frac{\sqrt{3}}{4}|v|}b \bigg( \frac{v-v_1}{\vert v-v_1 \vert} \cdot \sigma \bigg).    
\end{equation}
 
We will estimate separately for $|v_1|>|v|/4$ and for $|v_1|\leq |v|/4$. First, denoting $\widetilde{g}(v)=\l v\r g(v)$, $\widetilde{h}(v)=\l v\r h(v)$, we have
\begin{align}
\l v\r^l\int_{\mathbb{R}^3 \times \S^2}& \mathds{1}_{|v_1|>|v|/4} \vert v-v_1 \vert \vert f(v_1) \vert \vert  g(v^{*}) \vert \vert h(v_1^{*}) \vert b\big(\frac{v-v_1}{\vert v-v_1 \vert} \cdot \sigma \big) dv_1 d\sigma\nonumber  \\
&\lesssim \int_{\mathbb{R}^3 \times \S^2} \vert v-v_1 \vert \vert \l v_1\r^l f(v_1) \vert \vert  g(v^{*}) \vert \vert h(v_1^{*}) \vert b\big(\frac{v-v_1}{\vert v-v_1 \vert} \cdot \sigma \big) dv_1 d\sigma\nonumber\\
&\leq \|\l v\r^l f\|_{L^\infty_v}\int_{\R^3\times\S^2}\frac{|v-v_1|}{\l v^*\r\l v_1^*\r}| \widetilde{g}(v^*)||\widetilde{h}(v_1^*)|b\big(\frac{v-v_1}{\vert v-v_1 \vert} \cdot \sigma \big)\,dv_1\,d\sigma\nonumber\\
&\leq \|\l v\r^l f\|_{L^\infty_v}\|\widetilde{g}\|_{L^\infty}\|\widetilde{h}\|_{L^1_v}\label{truncation estimate 1},
\end{align}
where the last line is obtained in the same way as estimate \eqref{G_2 weighted estimate}.

Now using \eqref{comparing characteristics}, we also have
\begin{align}
\l v\r^l\int_{\mathbb{R}^3 \times \S^2}& \mathds{1}_{|v_1|\leq |v|/4} \vert v-v_1 \vert \vert f(v_1) \vert \vert  g(v^{*}) \vert \vert h(v_1^{*}) \vert b\big(\frac{v-v_1}{\vert v-v_1 \vert} \cdot \sigma \big) dv_1 d\sigma\nonumber\\
 &\leq \l v\r^l\int_{\mathbb{R}^3 \times \S^2} \mathds{1}_{|v^*|\geq\frac{\sqrt{3}}{4}|v|} \vert v-v_1 \vert \vert f(v_1) \vert \vert  g(v^{*}) \vert \vert h(v_1^{*}) \vert  dv_1 d\sigma\nonumber\\
 &\lesssim \int_{\mathbb{R}^3 \times \S^2} \vert v-v_1 \vert \vert  f(v_1) \vert \vert  \l v^*\r^l g(v^{*}) \vert \vert \l h(v_1^{*}) \vert  dv_1 d\sigma\nonumber\\
 &\leq \|\l v\r^l g\|_{L^\infty_v} \int_{\R^3\times\S^2}\frac{|v-v_1|}{\l v_1^*\r^3}|f(v_1)|\l v_1^*\r^3|h(v_1^*)|\,d\sigma\,dv_1\nonumber\\
 &\leq  \|\l v\r^l g\|_{L^\infty_v}\|\l v\r^3 h\|_{L^\infty_v}\int_{\R^3}\vert f(v_1) \vert \int_{\S^2}\frac{|v-v_1|}{\l v_1^*\r^3}\,d\sigma\,dv_1\nonumber\\
 &\lesssim
 \|f\|_{L^1_v}\|\l v\r^l g\|_{L^\infty_v}\|\l v\r^3 h\|_{L^\infty_v},\label{truncation estimate 2}
\end{align}
where to obtain the last line we used Lemma \ref{singint}.

Combining \eqref{truncation estimate 1}-\eqref{truncation estimate 2}, the result follows.

\end{proof}

\begin{lemma}\label{lemma post using pre}
For $l\geq 0$, there hold the estimates
\begin{align} \label{G1 L1 estimate}
\Vert \langle v \rangle^l \mathcal{G}_1[f,g,h] \Vert_{L^{1}_v} \lesssim \Vert \mathcal{L}_2[\langle v \rangle^l \vert h \vert, \vert g \vert,\vert f \vert] \Vert_{L^1_{v}} + \Vert \mathcal{L}_2[\vert h \vert,\langle v \rangle^l \vert g \vert , \vert f \vert] \Vert_{L^1_{v}}
\end{align}
and
\begin{align} \label{G2 L1 estimate}
    \Vert \langle v \rangle^l \mathcal{G}_2[f,g,h] \Vert_{L^{1}_v} \lesssim \Vert \mathcal{L}_1[\vert h \vert,\vert g \vert,\langle v \rangle^l \vert f \vert] \Vert_{L^1_v}.
\end{align}
Moreover
\begin{align}\label{G1 estimate joint}
    \Vert \l v\r^l\mathcal{G}_1[f,g,h] \Vert_{L^1_{v}} \lesssim  \Vert f\Vert_{L^1_v}\Vert \l v\r^l g \Vert_{L^1_v}\Vert \langle v \rangle^3 h \Vert_{L^{\infty}_v} + \Vert \langle v \rangle^3 f \Vert_{L^\infty_v} \Vert \langle v \rangle g \Vert_{L^1_v}\Vert \langle v \rangle^{l} h \Vert_{L^{1}_v},
\end{align}


\end{lemma}

\begin{proof}   We first treat the case $l=0$. We perform the involutionary change of variables $(v,v_1)\to (v_1^*,v^*)$, to obtain
\begin{align*}
\|\mathcal{G}_1[\vert f \vert, \vert g \vert, \vert h \vert]\|_{L^1_v}&=\int_{\R^6\times\S^2}|v-v_1| \vert f(v_1)\vert \vert g(v^*) \vert \vert h(v_1^*) \vert b \left(\frac{v-v_1}{|v-v_1|}\cdot\sigma\right)\,d\sigma\,dv_1\,dv\\
&=\int_{\R^6\times\S^2}|v_1^*-v^*| \vert f(v^*) \vert \vert g(v_1) \vert h(v) \vert b\left(\frac{v_1^*-v^*}{|v_1^*-v^*|}\cdot\sigma\right)\,d\sigma\,dv_1\,dv
\end{align*}
Now, using \eqref{conservation of relative velocities}-\eqref{microreversibilty},
we obtain
\begin{align*}
\|\mathcal{G}_1[\vert f \vert , \vert g \vert, \vert h \vert]\|_{L^1_v}&=\int_{\R^6\times\S^2}|v-v_1| \vert h(v) \vert \vert g(v_1) \vert \vert f(v^*) \vert b\left(\frac{v-v_1}{|v-v_1|}\cdot\sigma\right)\,d\sigma\,dv_1\,dv\\
&=\|\mathcal{L}_2[\vert h \vert, \vert g \vert , \vert f \vert]\|_{L^1_v}.
\end{align*}
Similarly, we obtain $\|\mathcal{G}_2[\vert f \vert , \vert g \vert , \vert h \vert]\|_{L^1_v}=\|\mathcal{L}_1[\vert h \vert,\vert g \vert, \vert f \vert]\|_{L^1_v}$. 

Estimates \eqref{G1 L1 estimate}-\eqref{G2 L1 estimate} are proved and estimates \eqref{G1 estimate joint}  follows from  \ref{trilossmain}.

Now for  $l>0$, conservation of energy \eqref{conservation of energy} implies
$$\l v\r^l\lesssim \l v^*\r^l +\l v_1^*\r^l.$$
Applying the triangle inequality  as well, we obtain 
$$\|\l v\r^l \mathcal{G}_1[f,g,h]\|_{L^1_{v}}\leq \|\mathcal{G}_1[\vert f \vert ,\l v\r^l \vert g \vert, \vert h \vert]\|_{L^1_v}+\|\mathcal{G}_1[\vert f \vert, \vert g \vert,\l v\r^l \vert h \vert]\|_{L^1_v},$$
and estimates \eqref{G1 estimate joint}, \eqref{G1 L1 estimate} follow from the case $l=0$.
Estimate  \eqref{G2 L1 estimate} follows from the fact that $\l v\r^l\mathcal{G}_2[f,g,h]=\mathcal{G}_2[\l v\r^l f,g,h]$ and the case $l=0$.
\end{proof}

\begin{lemma}
We have for $l \geq 0$
\begin{align} \label{G_1 L2 estimate}
    \Vert \langle v \rangle^l \mathcal{G}_1[f,g,h] \Vert_{L^{2}_v} \lesssim \Vert f \Vert_{L^{1}_v} \Vert \langle v \rangle^l g \Vert_{L^2_v} \Vert \langle v \rangle^{3/2} h \Vert_{L^{\infty}_v} + \Vert \langle v \rangle^l f \Vert_{L^{1}_v} \Vert g \Vert_{L^2_v} \Vert \langle v \rangle^{3/2} h \Vert_{L^{\infty}_v}.
\end{align}

\end{lemma}
\begin{proof}
Assume first $l=0.$ We argue by duality. Let $\phi \in L^2.$ We have
\begin{align*}
   I:= & \int_{\mathbb{R}^3} \vert \phi(v) \vert \int_{\mathbb{R}^3 \times \S^2} \vert f(v_1) \vert \vert g(v^{*}) \vert \vert h(v_1^{*}) \vert \vert v-v_1 \vert b\big( \frac{v-v_1}{\vert v-v_1 \vert} \cdot \sigma \big) d\sigma dv_1 dv \\
 \lesssim & \Vert f \Vert_{L^1_v} \Bigg \Vert \int_{\mathbb{R}^3} \vert \phi(v) \vert \int_{\S^2} \vert g(v^{*}) \vert \vert h(v_1^{*}) \vert \vert v-v_1 \vert b \big( \frac{v-v_1}{\vert v-v_1 \vert} \cdot \sigma \big) d\sigma  dv \Bigg \Vert_{L^{\infty}_{v_1}} \\
 \lesssim & \Vert f \Vert_{L^1_v} \Vert \phi \Vert_{L^2_v} \Bigg \Vert \int_{\S^2} \vert g(v^*) \vert \vert h(v_1^{*}) \vert  \vert v-v_1 \vert b \big( \frac{v-v_1}{\vert v-v_1 \vert} \cdot \sigma \big) d\sigma \Bigg \Vert_{L^{\infty}_{v_1} L^2_v}  ,
\end{align*}
where for the last line we used the Cauchy-Schwarz inequality in $v.$ 

Next we use the Cauchy-Schwarz inequality in $\sigma$ to write
\begin{align*}
   &  \int_{\S^2} \vert g(v^{*}) \vert \vert h(v_1^{*}) \vert \vert v-v_1 \vert b \big( \frac{v-v_1}{\vert v-v_1 \vert} \cdot \sigma \big) d\sigma \\  \leq & \Vert \langle v \rangle^{3/2} h \Vert_{L^{\infty}_v} \bigg( \int_{\S^2} g^2 (v^{*}) b \bigg( \frac{v-v_1}{\vert v-v_1 \vert} \cdot \sigma \bigg) d\sigma \bigg)^{1/2} \bigg( \int_{\S^2} \frac{\vert v-v_1 \vert^2}{\langle v_1^{*} \rangle^3}  b \bigg( \frac{v-v_1}{\vert v-v_1 \vert} \cdot \sigma \bigg) d\sigma \bigg)^{1/2}.
\end{align*}
Squaring this estimate, integrating in $v$ and using \eqref{estimate for v pre} applied to $\psi = g^2$ we find 
\begin{align*}
I & \lesssim \Vert f \Vert_{L^1_v} \Vert \phi \Vert_{L^2_v} \Vert \langle v \rangle^{3/2} h \Vert_{L^{\infty}_v} \Vert g \Vert_{L^2_v} \Bigg \Vert \int_{\S^2} \frac{\vert v-v_1 \vert^2}{\langle v_1 ^{*} \rangle^3} b \bigg( \frac{v-v_1}{\vert v-v_1 \vert} \cdot \sigma \bigg)  d\sigma \Bigg \Vert_{L^{\infty}_{v,v_1}},
\end{align*}
which yields the desired result with Lemma \ref{singint} applied to $\gamma=2.$ 

Now for  $l>0$, inequality \eqref{comparing characteristics} implies

$$ b \bigg( \frac{v-v_1}{\vert v-v_1 \vert} \cdot \sigma \bigg)\leq\left(\mathds{1}_{|v_1|>|v|/4}+\mathds{1}_{|v^*|\geq\frac{\sqrt{3}}{4}|v|}\right)b \bigg( \frac{v-v_1}{\vert v-v_1 \vert} \cdot \sigma \bigg) .$$
Using the triangle inequality as well, we obtain
$$\|\l v\r^l\mathcal{G}_1[f,g,h]\|_{L^2_v}\lesssim \|\mathcal{G}_1[\l v\r^l \vert f \vert ,\vert g \vert , \vert h \vert ] \Vert_{L^2_v}+\|\mathcal{G}_1[\vert f \vert,\l v\r^l \vert g \vert, \vert h \vert]\|_{L^2_v},$$
and the result follows from the case $l=0$.
\end{proof}

\subsection{Summary}
In this section we record straightforward corollaries of the estimates obtained above, in a form better adapted to the analysis of \eqref{KWE} in the remainder of the paper. 

\begin{corollary}
Let $\mathcal{T} \in \lbrace \mathcal{L}_1, \mathcal{L}_2, \mathcal{G}_1, \mathcal{G}_2 \rbrace$ and $l\ge 0$. The following estimates hold

\begin{align} \label{joint weighted Linfty}
    \Vert \langle v \rangle^l \mathcal{T}[f,f,f] \Vert_{L^{\infty}_v} \lesssim \Vert \langle v \rangle^3 f \Vert_{L^{\infty}_v} \Vert \langle v \rangle^l f \Vert_{L^{\infty}_v} \Vert \langle v \rangle f \Vert_{L^1_v}, 
\end{align}
\begin{align} \label{joint L2 1w}
     \Vert \langle v \rangle^l \mathcal{T}[f,f,f] \Vert_{L^{2}_v} \lesssim \left(\Vert \langle v \rangle^l f \Vert_{L^2_v} \Vert \langle v \rangle f \Vert_{L^1_v} +\|\l v\r^l f\|_{L^1_v}\| f\|_{L^2_v}\right)\|\l v\r^3 f\|_{L^\infty_v},
\end{align}
\begin{align} \label{joint L1 lw}
    \Vert \langle v \rangle^l \mathcal{T}[f,f,f] \Vert_{L^{1}_v} \lesssim \Vert \langle v \rangle^l f \Vert_{L^1_v} \Vert \langle v \rangle f \Vert_{L^1_v} \Vert \langle v \rangle^{3} f \Vert_{L^{\infty}_v}.
\end{align}
Finally, there holds the non-weighted estimate
\begin{align} \label{joint L1}
    \Vert \mathcal{T}[f,f,f] \Vert_{L^{1}_v} \lesssim \Vert f \Vert_{L^1_v} \Vert \langle v \rangle f \Vert_{L^1_v} \Vert \langle v \rangle f \Vert_{L^{\infty}_v}.
\end{align}

\end{corollary}
\begin{proof}
Estimate \eqref{joint weighted Linfty} follows from \eqref{trilossmain}, \eqref{G_2 weighted estimate} and \eqref{G1 weighted estimate}. Estimate \eqref{joint L2 1w} follows from \eqref{trilossmain}, \eqref{G_2 weighted estimate} and \eqref{G_1 L2 estimate}. Estimate \eqref{joint L1 lw}  follows from \eqref{G1 L1 estimate}-\eqref{G2 L1 estimate} and \eqref{trilossmain}.
Estimate \eqref{joint L1} is a consequence of \eqref{G1 L1 estimate}-\eqref{G2 L1 estimate} and \eqref{triloss1}-\eqref{triloss2}.
\end{proof}

Finally we give estimates that will be useful for contraction arguments in the next sections. We start with estimates that will be needed when solving the final state problem in the scattering section. 
\begin{lemma} \label{contractscatt}
The following estimates holds for $\mathcal{T} \in \lbrace \mathcal{L}_1, \mathcal{L}_2, \mathcal{G}_1, \mathcal{G}_2 \rbrace,$
\begin{equation}\label{contraction estimate L1}
\begin{aligned}
& \Vert \mathcal{T}[f,h_1,h_2] \Vert_{L^1_{v}} + \Vert  \mathcal{T}[h_1,f,h_2] \Vert_{L^1_{v}} + \Vert \mathcal{T}[h_1,h_2,f] \Vert_{L^1_{v}}  \\
\lesssim & \Vert f \Vert_{L^1_{v}} \big( \Vert \l v \r h_1 \Vert_{L^1_v} \Vert \l v \r^3 h_2 \Vert_{L^\infty_{v}} + \Vert \l v \r h_2 \Vert_{L^1_v} \Vert \l v \r^3 h_1 \Vert_{L^\infty_{v}} \big).
\end{aligned}
\end{equation}
Moreover for $l \geq 0$
\begin{equation} \label{contraction estimate Linfty}
    \begin{aligned}
        & \Vert \langle v \rangle^l \mathcal{T}[f,h_1,h_2] \Vert_{L^\infty_{v}} + \Vert \langle v \rangle^l \mathcal{T}[h_1,f,h_2] \Vert_{L^\infty_{v}} + \Vert \langle v \rangle^l \mathcal{T}[h_1,h_2,f] \Vert_{L^\infty_{v}}  \\
\lesssim & \Vert \langle v \rangle^l f \Vert_{L^\infty_{v}} \big( \Vert \l v \r h_1 \Vert_{L^1_v} \Vert \l v \r^3 h_2 \Vert_{L^\infty_{v}} + \Vert \l v \r h_2 \Vert_{L^1_v} \Vert \l v \r^3 h_1 \Vert_{L^\infty_{v}} \big) \\
+ &  \Vert \langle v \rangle f \Vert_{L^1_{v}} \big( \Vert \l v \r^l h_1 \Vert_{L^\infty_{v}} \Vert \l v \r^3 h_2 \Vert_{L^\infty_{v}} +\Vert \l v \r^3 h_1 \Vert_{L^\infty_{v}} \Vert \l v \r^l h_2 \Vert_{L^\infty_{v}} \big) \\
+ & \Vert \langle v \rangle^3 f \Vert_{L^\infty_{v}} \big( \Vert \l v \r h_1 \Vert_{L^1_v} \Vert \l v \r^l h_2 \Vert_{L^\infty_{v}} + \Vert \l v \r h_2 \Vert_{L^1_v} \Vert \l v \r^l h_1 \Vert_{L^\infty_{v}} \big) .
    \end{aligned}
\end{equation} 
Finally 
\begin{equation} \label{contraction estimate L2}
    \begin{aligned}
        & \Vert \langle v \rangle \mathcal{T}[f,h_1,h_2] \Vert_{L^2_{v}} + \Vert \langle v \rangle \mathcal{T}[h_1,f,h_2] \Vert_{L^2_{v}} + \Vert \langle v \rangle \mathcal{T}[h_1,h_2,f] \Vert_{L^2_{v}}  \\
\lesssim & \Vert \langle v \rangle f \Vert_{L^2_{v}} \big( \Vert \l v \r h_1 \Vert_{L^1_v} \Vert \l v \r^3 h_2 \Vert_{L^\infty_{v}} + \Vert \l v \r h_2 \Vert_{L^1_v} \Vert \l v \r^3 h_1 \Vert_{L^\infty_{v}} \big) \\
+ & \Vert \langle v \rangle f \Vert_{L^1_{v}} \big( \Vert \l v \r h_1 \Vert_{L^2_v} \Vert \l v \r^3 h_2 \Vert_{L^\infty_{v}} + \Vert \l v \r h_2 \Vert_{L^2_v} \Vert \l v \r^3 h_1 \Vert_{L^\infty_{v}} \big) \\
+ & \Vert \langle v \rangle^3 f \Vert_{L^\infty_{v}} \big( \Vert \l v \r h_1 \Vert_{L^1_v} \Vert \l v \r h_2 \Vert_{L^2_{v}} + \Vert \l v \r h_2 \Vert_{L^1_v} \Vert \l v \r h_1 \Vert_{L^2_{v}} \big).
    \end{aligned}
\end{equation} 
\end{lemma}
\begin{proof}
    For estimate \eqref{contraction estimate L1},  first note that given \eqref{G1 L1 estimate}, \eqref{G2 L1 estimate} we have 
\begin{align*}
    & \sum_{\mathcal{T} \in \lbrace \mathcal{L}_1, \mathcal{L}_2, \mathcal{G}_1,\mathcal{G}_2 \rbrace} \Vert \mathcal{T}[f,h_1,h_2] \Vert_{L^1_{xv}} + \Vert  \mathcal{T}[h_1,f,h_2] \Vert_{L^1_{xv}} + \Vert \mathcal{T}[h_1,h_2,f] \Vert_{L^1_{xv}} \\
    \lesssim & \sum_{\mathcal{T} \in \lbrace \mathcal{L}_1, \mathcal{L}_2 \rbrace} \Vert \mathcal{T}[f,h_1,h_2] \Vert_{L^1_{xv}} + \Vert  \mathcal{T}[h_1,f,h_2] \Vert_{L^1_{xv}} + \Vert \mathcal{T}[h_1,h_2,f] \Vert_{L^1_{xv}}. 
\end{align*}
This reduces the problem to proving the desired estimate for the loss operators only. 

In the case of $\mathcal{L}_1, \mathcal{L}_2$ we use \eqref{trilossmain} if $f$ is in the first or second position. Otherwise we use \eqref{triloss1}, \eqref{triloss2}.

The second estimate \eqref{contraction estimate Linfty} follows from \eqref{trilossmain} for the loss terms, \eqref{G1 weighted estimate} for the first gain term and \eqref{G_2 weighted estimate} for the second.

The last estimate follows from \eqref{trilossmain} for the loss terms, \eqref{G_1 L2 estimate} for the first gain term and \eqref{G_2 weighted estimate} for the second. 
\end{proof}
Finally we record a contraction estimate that will be useful for the local well-posedness theorem proved below for \eqref{KWE}.
\begin{lemma} \label{contractLWP}
The following estimate holds for $\mathcal{T} \in \lbrace \mathcal{L}_1, \mathcal{L}_2, \mathcal{G}_1, \mathcal{G}_2 \rbrace,$ and $l>4$
\begin{equation}\label{contraction estimate Linf}
\begin{aligned}
\ \Vert \l v\r^l\mathcal{T}[h_1,h_2,h_3] \Vert_{L^\infty_{v}} 
\lesssim  \|\l v\r^{l}h_1\|_{L^\infty_v}\|\l v\r^{l}h_2\|_{L^\infty_v}\|\l v\r^{l}h_3\|_{L^\infty_v}.
\end{aligned}
\end{equation}
\end{lemma}
\begin{proof}
 This follows from \eqref{trilossmain}, \eqref{G_2 weighted estimate}-\eqref{G1 weighted estimate}, the inequality 
 $$\|\l v\r u\|_{L^1_v}\leq \|\l v\r^{1-l}\|_{L^1_v}\|\l v\r^{l}u\|_{L^\infty_v}, $$ 
 and the fact $l>4$.
\end{proof}

\section{Global dispersive solutions} \label{section:GWP}
In this section we construct global dispersive solutions to \eqref{KWE}. Since the argument is perturbative, the free transport plays a fundamental role in studying the full nonlinear problem. We start by recalling some of its basic properties. We also define the functional analytic spaces used in our proof. Then we show local existence in a weighted $L^\infty$ space. To continue our solutions for all times, we rely on a bootstrap argument, which boils down to proving dispersive, amplitude and integrability bounds. We conclude by showing that our analysis also yields propagation of $L^1_{xv}$ -- moments (even though these are not needed in the global existence proof), and conservation in the special cases of mass, momentum and energy.

\subsection{Free transport}
The free transport semi-group is defined as follows:
\begin{definition}[Free transport semi-group] Let $f:\R_x^3\times\R_v^3\to\R$.
We denote
\begin{align}\label{transport semigroup}
    \big(\mathcal{S}(t)f \big)(x,v) := f(x-vt , v). 
\end{align}
\end{definition}
\begin{remark}
It is evident that the linear flow is measure preserving. In particular, for any $f\in L^1(\R^3_x\times\R^3_v)$, there holds
\begin{equation}\label{measure preserving}
\int_{\R^3\times\R^3} \big( \mathcal{S}(t)f \big) (x,v)\,dx\,dv=\int_{\R^3\times\R^3}f(x,y)\,dx\,dv ,\quad\forall t\in\R.   
\end{equation}
\end{remark}
\begin{remark}
An obvious, yet important, property of the linear semigroup is that it commutes with velocity weights:
\begin{equation}\label{commutation identity}
\l v\r^l(\mathcal{S}(t)f(x,v))=\mathcal{S}(t)(\l v\r^lf(x,v)),\quad \forall \,t\in\R,\quad l\geq 0.
\end{equation}
\end{remark}
We now recall several well known properties of this semi-group, see e.g. \cite{Arsenio}.
They consist in dispersive estimates and conservation of weighted $L^p$ norms, which we prove for convenience of the reader.
\begin{proposition}
We have for $p \geq r \geq 1$
\begin{align} \label{decaylin}
    \Vert \mathcal{S}(t) f \Vert_{L^{p}_x L^r _v} \leq \vert t \vert ^{-3\big(\frac{1}{r} - \frac{1}{p} \big)} \Vert f \Vert_{L^r_x L^{p}_v}. 
\end{align}
Moreover for $1 \leq p \leq \infty$, $l\geq 0$, and all $t \in \mathbb{R}$, we have
\begin{align} 
 \label{conslinweight}  \Vert \langle v \rangle^l \mathcal{S}(t) f \Vert_{L^{p}_{xv}} &= \Vert  \mathcal{S}(t) \langle v \rangle^l f \Vert_{L^{p}_{xv}} = \Vert \langle v \rangle^l f \Vert_{L^{p}_{xv}}.
\end{align}
\end{proposition}
\begin{proof}
We first prove the dispersive bound \eqref{decaylin}. Consider $p\geq r\geq 1$ and $t \neq 0$. Then, given $x\in\R^3$, we change variables $x'=x-vt$ to obtain
$$\|\mathcal{S}(t)f(x,v)\|_{L^r_v}=\vert t \vert ^{-\frac{3}{r}}\left\|f\left(x',\frac{x-x'}{t}\right)\right\|_{L^r_{x'}}.$$
Hence, changing variables $v'=\frac{x-x'}{t}$, and using Minkowski's integral inequality, we obtain
\begin{align*}
\|\mathcal{S}(t)f(x,v)\|_{L^p_x L^r_v}
=\vert t \vert ^{-3\left(\frac{1}{r}-\frac{1}{p}\right)}\|f(x',v')\|_{L^p_{v'} L^r_{x'}}&\leq \vert t \vert ^{-3\left(\frac{1}{r}-\frac{1}{p}\right)}\|f(x',v')\|_{L^r_{x'}L^p_{v'}}\\
&=\vert t \vert ^{-3\left(\frac{1}{r}-\frac{1}{p}\right)}\|f(x,v)\|_{L^r_x L^p_v}.
\end{align*}
Now \eqref{conslinweight} follows from \eqref{commutation identity} and \eqref{measure preserving}.
\end{proof}

\subsection{Functional spaces and definition of a solution} 

For the rest of the paper let us fix $M>8$ and $\frac{5}{2}<\alpha<M-\frac{3}{2}$.

Motivated by the free transport estimates of the previous section, we define the space of our initial data 
\begin{equation} \label{defXM}
X_{M,\alpha}=\{f\in \l v \r^{-M} L^{\infty}_{xv} \cap L^1_{xv} :\|f\|_{X_{M,\alpha}}<\infty\},  
\end{equation}
where we denote
\begin{equation} \label{defnormXM}
\|f\|_{X_{M,\alpha}} :=    \Vert \l v \r^M f \Vert_{L^{\infty}_{xv}} + \Vert f \Vert_{L^{1}_{xv}}  + \sup_{t \in \R} \langle t \rangle^{3/2} \big( \Vert \l v \r^\alpha \mathcal{S}(t) f \Vert_{L^{\infty}_x L^2_v} + \Vert \l v \r \mathcal{S}(t) f \Vert_{L^2_x L^1_v} \big) .
\end{equation}
We note that $(X_{M,\alpha},\|\cdot\|_{X_{M,\alpha}})$ is a Banach space.
\begin{remark}
$X_{M,\alpha}$ contains $\langle v \rangle^{-1} L^1_x L^{2}_v \big(\mathbb{R}_x^3 \times \mathbb{R}_v^3 \big) \cap \langle v \rangle^{-\alpha} L^2_x L^{\infty}_v \big(\mathbb{R}_x^3 \times \mathbb{R}_v^3 \big) \cap \langle v \rangle^{-M} L^{\infty}_{xv} \big(\mathbb{R}_x^3 \times \mathbb{R}_v^3 \big) \cap L^{1}_{xv} \big(\mathbb{R}_x^3 \times \mathbb{R}_v^3 \big)$ which (taking $M=8^+$) itself contains $\langle x \rangle^{-3^{-}} \langle v \rangle^{-8^{-}}  L^{\infty}_{xv}.$
\end{remark}

With the definition of $X_{M,\alpha}$ in hand now give the definition of a solution to \eqref{KWE}
\begin{definition}\label{definition of a solution} Let $I\subseteq\R$ be an interval with $0\in I$, and $f_0\in X_{M,\alpha}$. We say that $f\in \mathcal{C}(I,\l v\r^{-M}L^\infty_{xv})$ is a solution of \eqref{KWE} in $I$ with initial data $f_0$, if
\begin{equation}\label{Duhamel formula}
f(t)=\mathcal{S}(t)f_0+\int_0^t\mathcal{S}(t-s)\mathcal{C}[f](s)\,ds,\quad\forall t\in I.    
\end{equation}  
If $I=\R$, we say that the solution $f$ is global.
\end{definition}
\begin{remark}\label{remark on well defined terms}
For $f\in \mathcal{C}(I,\l v\r^{-M}L^\infty_{xv})$, we have  
$$\int_0^t\mathcal{S}(t-s)\mathcal{C}[f](s)\,ds\in \l v\r^{-M}L^\infty_{xv},\quad\forall t\in I,$$
thus every term in \eqref{Duhamel formula} is a well-defined element of $\l v\r^{-M}L^\infty_{xv}$.
Indeed, for $t\in I$, \eqref{conslinweight} and \eqref{contraction estimate Linf} imply
\begin{align*}
\left\|\l v\r^M \int_0^t \mathcal{S}(t-s)\mathcal{C}[f](s) \,ds\right\|_{L^\infty_{xv}}  & \lesssim |t|\sup_{s\in I}\|\l v\r^M f(s)\|_{L^\infty_{xv}}^3<+\infty
\end{align*}
\end{remark}
\begin{remark}
As stressed in the introduction, we work with strong solutions in this paper. This differs from other works in the literature \cite{BGGL}, \cite{Am}, \cite{AmMiPaTa24} where, due to the functional analytic framework used, the less familiar concept of mild solution has to be introduced. 
\end{remark}

\subsection{Main result}
We now state the main result of this section. Given $\varepsilon_0>0$, let us denote
\begin{equation}\label{X ball}
    B_{X_{M,\alpha}}(\varepsilon_0):=\{f\in X_{M,\alpha}\,:\,\|f\|_{X_{M,\alpha}}\leq \varepsilon_0\},
\end{equation}
the closed ball of radius $\varepsilon_0$ around the origin.
\begin{theorem} \label{thmdispersive}
    There exists $\overline{\varepsilon} > 0$ and $A>6$ such that if $0<\varepsilon_0 < \overline{\varepsilon}$ and $f_0 \in B_{X_{M,\alpha}}(\varepsilon_0),$ then \eqref{KWE}  has a unique global solution $f\in\mathcal{C}(\R,\l v\r^{-M}L^\infty_{xv})$ which is dispersive, meaning it satisfies the bound 
    \begin{align} \label{global bounds}
        \sup_{t \in \mathbb{R}} \big\{ \langle t \rangle^{3/2} \Vert \langle v \rangle f(t) \Vert_{L^{2}_x L^1_v} + \langle t \rangle^{3/2} \Vert \langle v \rangle^\alpha f(t) \Vert_{L^{\infty}_x L^2_v} + \Vert f(t) \Vert_{L^{1}_{xv}} + \Vert \langle v \rangle^M f(t) \Vert_{L^{\infty}_{xv}} \big\} < A\varepsilon_0. 
    \end{align}
    Moreover the solution is continuous with respect to initial data in the $L^1_{xv}$ topology i.e. if $f,g\in \mathcal{C}(\R,\l v\r^{-M}L^\infty_{x,v})$ are the solutions corresponding to $f_0,g_0\in B_{X_{M,\alpha}}(\varepsilon_0)$, there holds the estimate
    \begin{align} \label{contindep}
  \sup_{t \in \R} \Vert  f(t) - g(t)  \Vert_{L^{1}_{xv}} & \leq 2 \Vert  f_0 - g_0 \Vert_{L^1_{xv}}.
\end{align}
\end{theorem}
\begin{remark}
We elected to only prove continuity with respect to initial data for the $L^1_{xv}$ norm for simplicity. Indeed the contraction estimate \eqref{contraction estimate L1} is the most straightforward.
\end{remark}

As a corollary of our analysis we show that these global solutions propagate $L^1_{xv}$ -- moments and conserve the total mass, momentum, and energy:
\begin{corollary} \label{propagation moments}
Let $\bar{\varepsilon}$ be as in the statement of Theorem \ref{thmdispersive} and consider $N>0$. There exists $\overline{\varepsilon}_2<\overline{\varepsilon}$ independent of $N$ such that if $0<\varepsilon_0<\overline{\varepsilon}_2$ and $f_0 \in B_{X_{M,\alpha}}(\varepsilon_0)$ is such that $\Vert \langle v \rangle^N f_0 \Vert_{L^1_{xv}} < \varepsilon_0,$ then the global solution $f$ given by Theorem \ref{thmdispersive} satisfies
\begin{align}\label{propagation of moments statement}
    \sup_{t \in \mathbb{R}} \Vert \langle v \rangle^N f(t) \Vert_{L^1_{xv}} <  2 \varepsilon_0.
\end{align}

Moreover this control can be upgraded to conservation in the case where the weights are $1,v$ or $\vert v \vert^2.$ That is, for $\phi \in \lbrace 1,v,\vert v \vert^2 \rbrace$ we have the conservation laws
\begin{align}\label{conservation laws}
    \int_{\mathbb{R}^6} \phi(v) f(t,x,v) \, dv dx = \int_{\mathbb{R}^6} \phi(v) f_0(x,v) \, dv dx 
\end{align}
for all $t \in \mathbb{R}.$

\end{corollary}

\begin{remark}
We treat the full nonlinearity perturbatively, thus it is also possible to prove a lower bound of order $\varepsilon_0$ for the moments. This cannot be done with methods traditionally used for the Boltzmann equation (as in \cite{HJKL} for example), which consist in combining control for the gain only equation with a Kaniel-Shinbrot iteration scheme. 
\end{remark}

\subsection{Local existence}

We start by proving a local existence result for \eqref{KWE}. Recall the notation introduced in \eqref{defnormXM}, \eqref{X ball}. 
\begin{theorem} \label{local existence}
 There exists $\overline{\varepsilon}_0>0$  such that for all $0< \varepsilon_0 < \overline{\varepsilon}_0,$ and $f_0 \in B_{X_{M,\alpha}}(\varepsilon_0)$, equation \eqref{KWE} with initial data $f_0$ has a unique solution $f\in\mathcal{C}([0;3],\l v\r^{-M}L^\infty_{x,v})$. Moreover, $f$ satisfies the dispersive estimate:
\begin{equation}\label{dispersive estimate local}
   \sup_{t \in [0;3]} \big\{ \langle t \rangle^{3/2} \Vert \langle v \rangle f(t) \Vert_{L^{2}_x L^1_v} + \langle t \rangle^{3/2} \Vert \langle v \rangle^\alpha f(t) \Vert_{L^{\infty}_x L^2_v} + \Vert f(t) \Vert_{L^{1}_{xv}} + \Vert \langle v \rangle^M f(t) \Vert_{L^{\infty}_{xv}} \big\} \leq A_0\varepsilon_0, 
\end{equation}
for some numerical constant $A_0>6$ independent of $\varepsilon_0$.
\end{theorem}

\begin{proof} 
Consider $\bar{\varepsilon}_0>0$. 
Given $0<\varepsilon_0<\bar{\varepsilon}_0$ and $f_0\in B_{X_{M,\alpha}}(\varepsilon_0)$, we define the mapping
$$
\Phi: f(t)  \longmapsto  \mathcal{S}(t) f_{0} + \int_0^{t} \mathcal{S}(t-s) \mathcal{C}[f] \, ds,\quad t\in[0,3].
$$  
We show that $\Phi$ is a contraction on 
$$\bm{B}:=\left\{f\in \mathcal{C}([0;3],\l v\r^{-M}L^\infty_{xv})\,:\, \sup_{t\in[0;3]}\|\l v\r^Mf(t)\|_{L^\infty_{xv}}\leq 2\varepsilon_0\right\},$$
as long as $\bar{\varepsilon}_0$ is small enough.
\newline
\newline
\underline{Stability:}\newline
Let $f\in \bm{B}$. Given $t\in[0,3]$,  the triangle inequality, followed by \eqref{conslinweight}  and the fact that $f_0\in B_{X_{M,\alpha}}(\varepsilon_0)$ yield
\begin{align}\label{Duhamel ineq L inf local}
\|\l v\r^M\Phi f (t)\|_{L^\infty_{xv}} 
&\leq \varepsilon_0+ \sum_{\mathcal{T}\in\{\mathcal{L}_1,\mathcal{L}_2,\mathcal{G}_1,\mathcal{G}_2\}}\int_0^t\|\l v\r^M\mathcal{T}[f,f,f](s)\|_{L^\infty_{xv}}\,ds.
\end{align}
Now, using  \eqref{contraction estimate Linf}, for $\mathcal{T}\in\{\mathcal{L}_1,\mathcal{L}_2,\mathcal{G}_1,\mathcal{G}_2\}$ we have
\begin{align*}
    \int_0^t\|\l v\r^M\mathcal{T}[f,f,f](s)\|_{L^\infty_{xv}}\,ds\lesssim\int_0^t\|\l v\r^M f(s)\|_{L^\infty_{xv}}^3\,ds\leq 3\sup_{s\in[0;3]}  \|\l v\r^M f(s)\|_{L^\infty_{xv}}^3 \lesssim\varepsilon_0^3 <\varepsilon_0,
\end{align*}
since $\varepsilon_0<\bar{\varepsilon}_0$ and $\bar{\varepsilon}_0$ is sufficiently small.
 Therefore by \eqref{Duhamel ineq L inf local}, we obtain
$$\sup_{t\in[0;3]}\|\l v\r^M\Phi f (t)\|_{L^\infty_{xv}}<2\varepsilon_0.$$

Moreover, $\Phi$ is clearly continuous on $[0;3]$, hence $\Phi:\bm{B}\to \bm{B}$.
\newline
\newline
\underline{Contraction:} \newline
Let $f,g \in \bm{B}$. Given $t\in[0;3]$, trilinearity of $\mathcal{L}_1,\mathcal{L}_2,\mathcal{G}_1,\mathcal{G}_2$ implies
\begin{align*}
   \big \vert \Phi f(t)& - \Phi g(t) \big \vert\\ 
    &\leq \sum_{\mathcal{T}\in\{\mathcal{L}_1,\mathcal{L}_2,\mathcal{G}_1,\mathcal{G}_2\}} \bigg \vert \int_0^t \mathcal{S}(t-s) \big( \mathcal{T}[f-g,f,f] + \mathcal{T}[g,f-g,f] + \mathcal{T}[g,g,f-g] \big)(s) \, ds \bigg \vert.
\end{align*}
Hence,    \eqref{conslinweight} followed by estimate \eqref{contraction estimate Linf} imply
\begin{align*}&\|\l v\r^M\left(\Phi f (t)-\Phi g(t)\right)\|_{L^\infty_{xv}}\\
&\lesssim \int_0^t \|\l v\r^M\left(f(s)-g(s)\right)\|_{L^\infty_{xv}}\\
&\quad\times\left(\|\l v\r^M f(s)\|_{L^\infty_{xv}}^2+\|\l v\r^M f(s)\|_{L^\infty_{xv}}\|\l v\r^M g(s)\|_{L^\infty_{xv}}+\|\l v\r^M g(s)\|_{L^\infty_{xv}}^2\right)\,ds\\
&\lesssim\varepsilon_0^2\|f-g\|_{\mathcal{C}([0;3],\l v \r^{-M} L_{x,v}^\infty)}<\frac{1}{2}\|f-g\|_{\mathcal{C}([0;3],\l v \r^{-M} L_{x,v}^\infty)},
\end{align*}
since $\varepsilon_0<\bar{\varepsilon}_0$ and $\bar{\varepsilon}_0$ is sufficiently small.

We conclude that $\Phi$ is a contraction of $\bm{B}$, hence the contraction mapping principle implies that \eqref{KWE} has a unique solution $f$ in $[0;3]$ with initial data $f_0$.

It remains to prove the dispersive bound \eqref{dispersive estimate local}.
By construction $f\in\bm{B}$, hence 
\begin{equation}\label{fisrst bound on Linf}
\sup_{t\in[0;3]}\|\l v\r^Mf(t)\|_{L^\infty_{xv}}\leq 2\varepsilon_0.    
\end{equation}
We now estimate $\|f(t)\|_{L^1_{xv}}$. For $t\in[0,3]$,  Duhamel's formula, the triangle inequality, \eqref{conslinweight}, and the fact that $f_0\in B_{X_{M,\alpha}}(\varepsilon_0)$   imply
\begin{align}\label{Duhamel ineq L1 local}
    \Vert f(t) \Vert_{L^1_{xv}}  
    & \leq \varepsilon_0 + \sum_{\mathcal{T} \in \lbrace \mathcal{L}_1, \mathcal{L}_2, \mathcal{G}_1, \mathcal{G}_2 \rbrace} \int_0^t \big \Vert \mathcal{T}[f,f,f](s) \big \Vert_{L^1_{xv}} \, ds
\end{align}
Hence, for $\mathcal{T} \in \lbrace \mathcal{L}_1, \mathcal{L}_2, \mathcal{G}_1, \mathcal{G}_2 \rbrace$, estimates \eqref{joint L1}, \eqref{fisrst bound on Linf} imply
\begin{align*}
 \int_0^t \big \Vert \mathcal{T}[f,f,f](s) \big \Vert_{L^1_{xv}} \, ds&\lesssim \int_0^t \|f(s)\|_{L^1_{xv}}\|\l v\r f(s)\|_{L^\infty_x L^1_v}\|\l v\r f(s)\|_{L^\infty_{xv}}\,ds\\
 &\leq 3\sup_{s\in[0;3]}\|f(s)\|_{L^1_{xv}}\|\l v\r f(s)\|_{L^\infty_x L^1_v}\|\l v\r f(s)\|_{L^\infty_{xv}}\\
    &\lesssim  \varepsilon_0^2 \sup_{s\in [0;3]} \Vert f(s) \Vert_{L^1_{xv}}<\frac{1}{2}\sup_{s\in[0;3]}\|f(s)\|_{L^1_{xv}},
\end{align*}
since $\varepsilon_0<\bar{\varepsilon}_0$ and $\bar{\varepsilon}_0$ is sufficiently small. To obtain the second to last line, we used the inequality $$\|\l v\r u\|_{L^1_v}\leq \|\l v\r^{-3^-}\|_{L^1_v}\|\l v\r^{4^+} u\|_{L^\infty_v}\lesssim\|\l v\r^{4^+} u\|_{L^\infty_v},$$
as well as \eqref{fisrst bound on Linf}.
This estimate combined with \eqref{Duhamel ineq L1 local} implies 
\begin{equation}\label{L1 bound local}
  \sup_{t\in[0;3]}\|f(t)\|_{L^1_{xv}} < 2\varepsilon_0. 
\end{equation}

In order to treat the term $\l t\r^{-3/2}\|\l v\r f(t)\|_{L^2_xL^1_v}$, we prove a slightly more general fact that will be useful later on. Let $\beta$ be such that $0 \leq \beta < \frac{M-3}{2}$. For an arbitrary measurable function $u:\R_x^3\times\R_v^3\to\R$,  we use Cauchy-Schwarz  in $v$ to write
\begin{align*}
    \Vert \langle v \rangle^{\beta} u \Vert_{L^{2}_x L^1_v} &\leq  \Vert \langle v \rangle^{-3/2^{-}} \Vert_{L^{2}_v} \Vert \langle v \rangle^{\beta + 3/2^{+}} u \Vert_{L^{2}_{xv}} 
\end{align*}
and
\begin{align}\label{L2xv}
\begin{split}
\Vert \langle v \rangle^{\beta + 3/2^{+}} u \Vert_{L^{2}_{xv}} &\lesssim  \Vert \langle v \rangle^{2 \beta+3^{+}} u  \Vert_{L^{\infty}_{xv}}^{1/2}   \Vert u \Vert_{L^{1}_{xv}}^{1/2} \leq \Vert \langle v \rangle^{M} u \Vert_{L^{\infty}_{xv}}^{1/2}   \Vert u \Vert_{L^{1}_{xv}}^{1/2}, 
   \end{split}
\end{align}
where for the last line we used the fact that $0\leq \beta<\frac{M-3}{2}$. 
Now using \eqref{fisrst bound on Linf}, \eqref{L1 bound local}, we obtain
\begin{align} \label{highweight}
    \Vert \langle v \rangle^\beta f(t) \Vert_{L^{2}_x L^1_v} \lesssim \varepsilon_0.
\end{align}
Finally, we have
\begin{align}\label{dispersive 2 local}
    \Vert \l v \r^\alpha f(t) \Vert_{L^\infty_x L^2_v} \leq \|\l v\r^{-3/2^-}\|_{L^2_v}\|\l v\r^{\alpha+3/2^+}f(t)\|_{L^\infty_{xv}}\lesssim \Vert \l v \r^{M} f(t) \Vert_{L^{\infty}_{xv}} \lesssim \varepsilon_0,
\end{align}
where we used the fact that $\alpha<M-\frac{3}{2}$ and  \eqref{fisrst bound on Linf}. 

Taking $A_0>6$ large enough, the result follows.

\end{proof}

\subsection{Global existence}

The extension of the solution to arbitrary large times relies on a bootstrap argument. The main bootstrap  proposition is the following 

\begin{proposition} \label{boot main prop} Let $\bar{\varepsilon}_0$ and $A_0>6$ be as in the statement of Theorem \ref{local existence}. Then there is $0<\bar{\varepsilon}_1<\bar{\varepsilon}_0$ with the following property:

If $f$ is a solution of \eqref{KWE} in $[0,T]$ with initial data $f_0\in B_{X_{M,\alpha}}(\varepsilon_0)$, where $0<\varepsilon_0<\bar{\varepsilon}_1$ and $T>3$, satisfying the bounds 
\begin{align} 
\label{boot1} \sup_{t \in [0;T]} \langle t \rangle^{3/2} \Vert \langle v \rangle f(t) \Vert_{L^{2}_x L^1_v} <  A_0^2\varepsilon_0, \\ 
\label{boot2} \sup_{t \in [0;T]} \langle t \rangle^{3/2} \Vert \langle v \rangle^\alpha f(t) \Vert_{L^{\infty}_x L^2_v} < A_0^2\varepsilon_0, \\ 
\label{boot3} \sup_{t \in [0;T]}  \Vert f(t) \Vert_{L^{1}_{xv}} <A_0^2\varepsilon_0, \\
\label{boot4} \sup_{t \in [0;T]} \Vert \langle v \rangle^M f(t) \Vert_{L^{\infty}_{xv}} <A_0^2\varepsilon_0,
\end{align} 
then the improved estimates
\begin{align} 
\label{boot1imp} \sup_{t \in [0;T]} \langle t \rangle^{3/2} \Vert \langle v \rangle f(t) \Vert_{L^{2}_x L^1_v} < A_0\varepsilon_0, \\ 
\label{boot2imp} \sup_{t \in [0;T]} \langle t \rangle^{3/2} \Vert \langle v \rangle^\alpha f(t) \Vert_{L^{\infty}_x L^2_v} < A_0\varepsilon_0, \\ 
\label{boot3imp} \sup_{t \in [0;T]}  \Vert f(t) \Vert_{L^{1}_{xv}} <A_0\varepsilon_0, \\
\label{boot4imp} \sup_{t \in [0;T]} \Vert \langle v \rangle^M f(t) \Vert_{L^{\infty}_{xv}} <A_0\varepsilon_0,
\end{align} 
hold.
\end{proposition}
\begin{proof} 
 Consider $0<\bar{\varepsilon}_1<\bar{\varepsilon}_0$ sufficiently small. 
 Let $0<\varepsilon_0<\bar{\varepsilon}_1$, and consider $f_0\in B_{X_{M,\alpha}}(\varepsilon_0)$. Due to the uniqueness part of Theorem \ref{local existence}, for $0\leq t\leq 3$ the solution $f$ coincides with the one obtained in Theorem \ref{local existence}, thus it satisfies \eqref{dispersive estimate local}. Consequently \eqref{boot1imp}-\eqref{boot4imp} are automatically satisfied.

Hence, it suffices to prove \eqref{boot1imp}-\eqref{boot4imp} for $3< t\leq T$ under the assumptions \eqref{boot1}-\eqref{boot4}.
The next four subsections are dedicated to the proof of each of these improved estimates.
In the following,  given $\mathcal{T}\in \{\mathcal{L}_1,\mathcal{L}_2,\mathcal{G}_1,\mathcal{G}_2\}$, we denote
\begin{align*}
    I_{\mathcal{T}}(t):= \int_0^t 
    \mathcal{S}(t-s) \mathcal{T}[f,f,f](s) \, ds, \quad 3<t\leq T.
\end{align*}

Let us also denote $\varepsilon:=A_0^2\varepsilon_0$.
\subsubsection{Proof of \eqref{boot1imp}} \label{subsection:boot1}
By Duhamel's formula, we obtain
\begin{align} \label{Duhamel}
\begin{split}
\big \Vert \langle v \rangle f(t) \big \Vert_{L^{2}_x L^1_v} \leq \Vert \langle v \rangle \mathcal{S}(t) f_0 \Vert_{L^{2}_x L^1_v} &+ \sum_{\mathcal{T} \in \lbrace \mathcal{G}_1, \mathcal{G}_2, \mathcal{L}_1,\mathcal{L}_2 \rbrace}  \Vert \langle v \rangle I_{\mathcal{T}}(t-1) \Vert_{L^{2}_x L^1_v} \\
&+\sum_{\mathcal{T} \in \lbrace \mathcal{G}_1, \mathcal{G}_2, \mathcal{L}_1,\mathcal{L}_2 \rbrace}   \Vert \langle v \rangle \big(I_{\mathcal{T}}(t) - I_{\mathcal{T}}(t-1) \big) \Vert_{L^{2}_x L^1_v}.
\end{split}
\end{align}
Fix $\mathcal{T}\in \{\mathcal{L}_1,\mathcal{L}_2,\mathcal{G}_1,\mathcal{G}_2\}$.
Using the commutation identity \eqref{commutation identity} as well as the dispersive estimate \eqref{decaylin} we find
\begin{align*}
    \Vert \langle v \rangle I_{\mathcal{T}}(t-1) \Vert_{L^{2}_x L^1_v} & \lesssim \int_{0}^{t-1} \frac{1}{(t-s)^{3/2}} \big \Vert \langle v \rangle \mathcal{T}[f,f,f](s) \big \Vert_{L^1_x L^{2}_v} ds.
\end{align*}
Now using \eqref{joint L2 1w} and the Cauchy-Schwarz inequality in $x$,  we obtain
\begin{align*}
\notag \big \Vert \langle v \rangle \mathcal{T}[f,f,f](s) \big \Vert_{L^1_x L^{2}_v} & \lesssim \Vert \langle v \rangle f(s) \Vert_{L^{2}_{xv}} \Vert \langle v \rangle f(s) \Vert_{L^2_{x} L^{1}_v} \Vert \langle v \rangle^3 f(s) \Vert_{L^{\infty}_{xv}}  \lesssim \langle s \rangle ^{-3/2} \varepsilon^3,\quad s\in [0,t-1].
\end{align*}
where we used \eqref{L2xv} and \eqref{boot1}, \eqref{boot4}.
Therefore,
\begin{align} 
\Vert \langle v \rangle I_{\mathcal{T}}(t-1) \Vert_{L^{2}_x L^1_v} &\lesssim \int_0^1 \frac{\varepsilon^3}{(t-s)^{3/2}} ds + \int_1^{t-1} \frac{\varepsilon^3}{(t-s)^{3/2} s^{3/2}}
ds\nonumber\\
&\leq \varepsilon^3\left((t-1)^{-3/2}+t^{-3/2}\right)\lesssim \langle t \rangle ^{-3/2} \varepsilon^3<\frac{A_0\varepsilon_0}{3}\l t\r^{-3/2} ,\label{boot1dispbulk}
\end{align}
since $t>3$, $\varepsilon_0<\bar{\varepsilon}_1$, and $\bar{\varepsilon}_1$ is sufficiently small.

For the term $I_{\mathcal{T}}(t)-I_{\mathcal{T}}(t-1)$, consider $\frac{3}{2}<q \leq \min \lbrace  \alpha-1, \frac{M-5}{2} \rbrace$. We note that the interval for $q$ is non-trivial because $\alpha>5/2$ and $M>8$.  Using the Cauchy-Schwarz inequality in $v$, the commutation relation \eqref{commutation identity} and \eqref{conslinweight}, we obtain
\begin{align*}
    \Vert \langle v \rangle \big( I_{\mathcal{T}}(t)-I_{\mathcal{T}}(t-1) \big) \Vert_{L^{2}_x L^1_v} & \leq  \int_{t-1}^t \|\l v\r^{-q}\|_{L^2_v}\Vert \langle v \rangle^{1+q} \mathcal{S}(t-s) \mathcal{T}[f,f,f](s) \Vert_{L^{2}_{xv}} \, ds  \\
    & \lesssim \int_{t-1}^t \Vert \langle v \rangle^{1+q} \mathcal{T}[f,f,f](s) \Vert_{L^{2}_{xv}} \, ds.
\end{align*}

Then using \eqref{joint L2 1w}, for $s\in[t-1,t]$, we have
\begin{align*} 
    \Vert &\langle v \rangle^{1+q} \mathcal{T}[f,f,f](s) \Vert_{L^{2}_{xv}}\\
    & \lesssim \Vert \langle v \rangle^{1+q} f(s) \Vert_{L^{\infty} _x L^2_{v}} \Vert \langle v \rangle f(s) \Vert_{L^{2}_x L^1_v}  \Vert \langle v \rangle^3 f(s) \Vert_{L^{\infty}_{xv}}+\Vert f(s) \Vert_{L^{\infty} _x L^2_{v}} \Vert \langle v \rangle^{1+q} f(s) \Vert_{L^{2}_x L^1_v}  \Vert \langle v \rangle^3 f(s) \Vert_{L^{\infty}_{xv}}.
\end{align*}
Using \eqref{boot2} (since $1+q \leq \alpha$), \eqref{boot1} and \eqref{boot4}, we have
$$\Vert \langle v \rangle^{1+q} f(s) \Vert_{L^{\infty} _x L^2_{v}} \Vert \langle v \rangle f(s) \Vert_{L^{2}_x L^1_v}  \Vert \langle v \rangle^3 f(s) \Vert_{L^{\infty}_{xv}}\lesssim\l s\r^{-3}\varepsilon^3.$$
Moreover, we have
$$\Vert f \Vert_{L^{\infty} _x L^2_{v}} \Vert \langle v \rangle^{1+q} f(s) \Vert_{L^{2}_x L^1_v}  \Vert \langle v \rangle^2 f(s) \Vert_{L^{\infty}_{xv}}\lesssim\l s\r^{-3/2}\varepsilon^3,$$
where we used \eqref{boot2} and \eqref{boot4} for the first and last term respectively, while for the second 
second term, we used \eqref{highweight} (since $1+q<\frac{M-3}{2}$).

Hence, we obtain
\begin{align} \label{boot1dispend}
   \big \Vert \langle v \rangle \big( I_{\mathcal{T}}(t)-I_{\mathcal{T}}(t-1) \big) \big \Vert_{L^{2}_x L^1_v} &\lesssim\varepsilon^3\int_{t-1}^t \l s\r^{-3/2}\,ds \leq \l t-1\r^{-3/2}\varepsilon^3\lesssim \l t\r^{-3/2} \varepsilon^3<\frac{A_0\varepsilon_0}{3}\l t\r^{-3/2},
\end{align}
since $t>3$,  $\varepsilon_0<\bar{\varepsilon}_1$, and $\bar{\varepsilon}_1$ is sufficiently small.
 
 Finally, by \eqref{commutation identity}-\eqref{decaylin} we have
\begin{align}\label{dispersive estimate proof 3.26}
    \Vert \langle v \rangle \mathcal{S}(t) f_0 \Vert_{L^{2}_x L^1_v} \leq t^{-3/2} \Vert \langle v \rangle f_0 \Vert_{L^1_x L^2_v}< \l t\r^{-3/2}\varepsilon_0<\frac{A_0\varepsilon_0}{3}\l t\r^{-3/2},
\end{align}
since $t>3$ and $A_0>6$.

Now putting together \eqref{Duhamel}-\eqref{dispersive estimate proof 3.26}, we obtain
    $\big \Vert \langle v \rangle f(t) \big \Vert_{L^{2}_x L^1_v} < \l t\r^{-3/2}A_0\varepsilon_0,
$
and \eqref{boot1imp} follows.

\subsubsection{Proof of \eqref{boot2imp}} \label{subsection:boot2} 
Again, by Duhamel's formula, we obtain
\begin{align} \label{Duhamel boot 2}
\begin{split}
\big \Vert \langle v \rangle^\alpha f(t) \big \Vert_{L^{\infty}_x L^2_v} & \leq \Vert \langle v \rangle^\alpha \mathcal{S}(t) f_0 \Vert_{L^{\infty}_x L^2_v} + \sum_{\mathcal{T} \in \lbrace \mathcal{G}_1, \mathcal{G}_2, \mathcal{L}_1,\mathcal{L}_2 \rbrace}  \Vert \langle v \rangle^\alpha I_{\mathcal{T}}(t-1) \Vert_{L^{\infty}_x L^2_v} \\
&+\sum_{\mathcal{T} \in \lbrace \mathcal{G}_1, \mathcal{G}_2, \mathcal{L}_1,\mathcal{L}_2 \rbrace}   \Vert \langle v \rangle^\alpha \big(I_{\mathcal{T}}(t) - I_{\mathcal{T}}(t-1) \big) \Vert_{L^{\infty}_x L^2_v}.
\end{split}
\end{align}
Fix $\mathcal{T} \in \lbrace \mathcal{L}_1, \mathcal{L}_2, \mathcal{G}_1, \mathcal{G}_2 \rbrace$.
First, we write
\begin{align*}
    \Vert \langle v \rangle^\alpha I_{\mathcal{T}}(t-1) \Vert_{L^{\infty}_x L^2_v} & \lesssim  \int_0^{t-1} \Vert \langle v \rangle^\alpha \mathcal{S}(t-s) \mathcal{T}[f,f,f] \Vert_{L^{\infty}_{x} L^2_x} \, ds.
\end{align*}
Using again \eqref{commutation identity}-\eqref{decaylin} we find 
\begin{align*}
\big \Vert \langle v \rangle^\alpha \mathcal{S}(t-s) \mathcal{T}[f,f,f](s) \big \Vert_{L^{\infty}_x L^2_v}  & = \big \Vert  \mathcal{S}(t-s) \langle v \rangle^\alpha \mathcal{T}[f,f,f](s) \big \Vert_{L^{\infty}_x L^2_v} \\
  & \lesssim \frac{1}{(t-s)^{3/2}} \Vert \langle v \rangle^\alpha \mathcal{T}[f,f,f](s) \Vert_{L^{2}_x L^{\infty}_v}.
\end{align*}
Now, \eqref{joint weighted Linfty} followed by \eqref{boot1}, \eqref{boot4} implies
\begin{align*} 
    \Vert \langle v \rangle^\alpha \mathcal{T}[f,f,f](s) \Vert_{L^{2}_x L^{\infty}_v} \lesssim \Vert \langle v \rangle^\alpha f(s) \Vert_{L^{\infty}_{xv}} \Vert \langle v \rangle f(s) \Vert_{L^2_x L^1_v} \Vert \langle v \rangle^3 f(s) \Vert_{L^{\infty}_{xv}} \lesssim \langle s \rangle^{-3/2} \varepsilon^3,
\end{align*}
which after estimating as in \eqref{boot1dispbulk} gives
\begin{align} \label{dispboot2bulk}
 \Vert \langle v \rangle^\alpha I_{\mathcal{T}}(t-1) \Vert_{L^{\infty}_x L^2_v} \lesssim \l t\r^{-3/2} \varepsilon^3<\frac{A_0\varepsilon_0}{3}\l t\r^{-3/2},
\end{align}
since $\varepsilon_0<\bar{\varepsilon}_1$, and $\bar{\varepsilon}_1$ is sufficiently small.

To deal with the term $I_{\mathcal{T}}(t)-I_{\mathcal{T}}(t-1)$, we bound as follows
\begin{align*}
   \Vert \langle v \rangle^\alpha \big(I_{\mathcal{T}}(t)-I_{\mathcal{T}}(t-1) \big) \Vert_{L^{\infty}_x L^2_v} & \lesssim \int_{t-1}^t \big \Vert \langle v \rangle^\alpha \mathcal{S}(t-s) \mathcal{T}[f,f,f](s) \big \Vert_{L^{\infty}_x L^2_v} \, ds \\
   & \lesssim \int_{t-1}^t \big \Vert \mathcal{S}(t-s) \langle v \rangle^{\alpha + 3/2^{+}} \mathcal{T}[f,f,f](s) \big \Vert_{L^{\infty}_{xv}} \, ds \\
   & \lesssim \int_{t-1}^t  \big \Vert \langle v \rangle^{\alpha + 3/2^{+}} \mathcal{T}[f,f,f](s) \big \Vert_{L^\infty_{xv}} \, ds, 
\end{align*}
where for the last line we used \eqref{conslinweight}.
Next, by \eqref{joint weighted Linfty} we obtain
\begin{align*}
  \big \Vert \langle v \rangle^{\alpha + 3/2^{+}} \mathcal{T}[f,f,f](s) \big \Vert_{L^\infty_{xv}} &\lesssim \Vert \langle v \rangle^{\alpha + 3/2^{+}} f(s) \Vert_{L^{\infty}_{xv}} \Vert \langle v \rangle f (s)\Vert_{L^{\infty}_x L^1_v} \Vert \langle v \rangle^3 f(s) \Vert_{L^{\infty}_{xv}} \\
  &\leq \varepsilon^2 \|\l v\r^{5/2^+} f(s)\|_{L^\infty_x L^2_v}\\
  &\lesssim \l s\r^{-3/2}\varepsilon^3,
\end{align*}
where for the second line we used Cauchy-Scharz in $v$ and \eqref{boot4} (since $\alpha<M-\frac{3}{2})$, and for the last line we used \eqref{boot2} (since $\alpha>5/2)$.
Hence 
\begin{align} \label{dispboot2end}
  \Vert \langle v \rangle^\alpha \big(I_{\mathcal{T}}(t)-I_{\mathcal{T}}(t-1) \big) \Vert_{L^{\infty}_x L^2_v} \lesssim \varepsilon^3\int_{t-1}^t \l s\r^{-3/2} \,ds\leq \l t-1\r^{-3/2}\varepsilon^3\lesssim \l t\r^{-3/2}\varepsilon^3<\frac{A_0\varepsilon_0}{3}\l t\r^{-3/2},
\end{align}
since $t>3$, $\varepsilon_0<\bar{\varepsilon}_1$, and $\bar{\varepsilon}_1$ is chosen sufficiently small.

Finally, similarly as in \eqref{dispersive estimate proof 3.26} we have
\begin{align}\label{dispersive estimate proof 3.27}
    \Vert \langle v \rangle^\alpha \mathcal{S}(t) f_0 \Vert_{L^{\infty}_x L^2_v} < \frac{A_0\varepsilon_0}{3}\l t\r^{-3/2},
\end{align}
since $t>3$ and $A_0>6$.
Now, \eqref{Duhamel boot 2}-\eqref{dispboot2end} imply 
$
    \Vert \langle v \rangle^\alpha f(t) \Vert_{L^{\infty}_x L^2_v} <  \langle t \rangle ^{-3/2}A_0\varepsilon_0,
$
and \eqref{boot2imp} follows.
\subsubsection{Proof of \eqref{boot3imp}} \label{subsection:boot3}
By Duhamel's formula,  \eqref{conslinweight},  and the fact that $f_0\in B_{X_{M,\alpha}}(\varepsilon_0)$ we find that
\begin{align} \label{Duhamelboot3}
    \Vert f(t) \Vert_{L^1_{xv}} & \leq \varepsilon_0  + \sum_{\mathcal{T} \in \lbrace \mathcal{G}_1, \mathcal{G}_2, \mathcal{L}_1, \mathcal{L}_2 \rbrace} \int_0^t \Vert \mathcal{T}[f,f,f](s) \Vert_{L^1_{xv}} \, ds.
\end{align}

Next using \eqref{joint L1} and we find for $\mathcal{T} \in \lbrace \mathcal{G}_1, \mathcal{G}_2, \mathcal{L}_1,\mathcal{L}_2 \rbrace$
\begin{align} 
    \Vert \mathcal{T}[f,f,f](s) \Vert_{L^1_{xv}} & \lesssim \Vert f (s)\Vert_{L^{1}_{xv}} \Vert \langle v \rangle f(s) \Vert_{L^\infty_x L^1_v} \Vert \langle v \rangle^3 f(s)\Vert_{L^{\infty}_{xv}}\nonumber  \\
   & \lesssim \Vert f(s) \Vert_{L^{1}_{xv}} \Vert \langle v \rangle^{5/2^{+}} f(s) \Vert_{L^\infty_x L^2_v} \Vert \langle v \rangle^3 f(s)\Vert_{L^{\infty}_{xv}}\nonumber\\
   & \lesssim \langle s \rangle^{-3/2} \varepsilon^3 ,\label{mainL1}
\end{align}
where for the second to last line  we used the Cauchy-Schwarz inequality in $v$, and for the last line we used  \eqref{boot3}, \eqref{boot2} (since $\alpha>5/2)$ and \eqref{boot4}.

Integrating \eqref{mainL1}, we obtain
\begin{equation}\label{mainL1int}
 \int_0^t \Vert \mathcal{T}[f,f,f](s) \Vert_{L^1_{xv}} \,ds\lesssim \varepsilon^3< \frac{A_0\varepsilon_0}{2},  
\end{equation}
since $\varepsilon_0<\bar{\varepsilon}_1$, and $\bar{\varepsilon}_1$ is sufficiently small.
Combining \eqref{Duhamelboot3}, \eqref{mainL1int} we find
\begin{align} \label{boot3concl}
    \Vert f(t) \Vert_{L^{1}_{xv}} <\varepsilon_0+\frac{A_0\varepsilon_0}{2}<A_0\varepsilon_0,
\end{align}
since $A_0>6$. Estimate \eqref{boot3imp} follows.
\subsubsection{Proof of \eqref{boot4imp}} \label{subsection:boot4}

Finally to improve the last bootstrap assumption, we use again Duhamel's formula,  \eqref{conslinweight},  and the fact that $f_0\in B_{X_{M,\alpha}}(\varepsilon_0)$ to write
\begin{align*}
& \Vert \langle v \rangle^M f(t) \Vert_{L^{\infty}_{xv}}  \leq \Vert \langle v \rangle^M f_0 \Vert_{L^{\infty}_{xv}} + \sum_{\mathcal{T} \in \lbrace \mathcal{G}_1, \mathcal{G}_2, \mathcal{L}_1, \mathcal{L}_2 \rbrace} \int_0^t \Vert \langle v \rangle^M \mathcal{T}[f,f,f] \Vert_{L^{\infty}_{xv}} \, ds .
\end{align*}
For fixed $\mathcal{T} \in \lbrace \mathcal{L}_1, \mathcal{L}_2, \mathcal{G}_1, \mathcal{G}_2 \rbrace$, we use  \eqref{joint weighted Linfty} and the Cauchy-Schwarz inequality in $v$ to obtain, 
\begin{align} \label{mainLinfty}
    \Vert \langle v \rangle^M \mathcal{T}[f,f,f](s) \Vert_{L^{\infty}_{xv}} & \lesssim \Vert \langle v \rangle^M f\Vert_{L^{\infty}_{xv}} \Vert \langle v \rangle f (s) \Vert_{L^{\infty}_x L^1_v} \Vert \langle v \rangle^3 f \Vert_{L^{\infty}_{xv}} \\
\notag    & \lesssim \Vert \langle v \rangle^M f\Vert_{L^{\infty}_{xv}} \Vert \langle v \rangle^{5/2^+} f (s) \Vert_{L^{\infty}_x L^2_v} \Vert \langle v \rangle^3 f \Vert_{L^{\infty}_{xv}}\\
  \notag  & \lesssim \langle s \rangle ^{-3/2} \varepsilon^3,
\end{align}
and we can conclude as in \eqref{boot3concl} that $\Vert \langle v \rangle^M f(t) \Vert_{L^{\infty}_{xv}}<A_0\varepsilon_0,$
since $\varepsilon_0<\bar{\varepsilon}_1$ and $\bar{\varepsilon}_1$ is sufficiently small. Estimate \eqref{boot4imp} follows.

The proof is  complete.

\end{proof}

\subsubsection{Conclusion of the proof}

Now with  Proposition \ref{boot main prop} in hand, we move to the proof of Theorem \ref{thmdispersive}. The argument is fairly standard, nonetheless we provide some details for the convenience of the reader.

\begin{proof}[Proof of Theorem \ref{thmdispersive}]

Let $\bar{\varepsilon}_1,A_0$ be as in the statement of Proposition \ref{boot main prop}.
Setting $\bar{\varepsilon}=\frac{\bar{\varepsilon}_1}{A_0}$, consider $0<\varepsilon_0<\bar{\varepsilon}$ and $f_0\in B_{X_{M,\alpha}}(\varepsilon_0)$. We define 
\begin{align*}
    T^*&:=\sup\Big\{T>0 : \text{ \eqref{KWE} has a solution $f$ with initial data $f_0$ in $[0;T]$, which satisfies \eqref{boot1imp}-\eqref{boot4imp}}\Big\}.
\end{align*}

Since $f_0\in B_{X_{M,\alpha}}(\varepsilon_0)$ and $\varepsilon_0<\frac{\bar{\varepsilon}_1}{A_0}<\bar{\varepsilon}_0$, Theorem \ref{local existence} implies that $T^*\geq 3$. Arguing by contradiction, we will show that $T^*=+\infty$. Indeed, assume that $T^*<+\infty$, and consider the solution $f$ of \eqref{KWE} with initial data $f_0$ satisfying \eqref{boot1imp}-\eqref{boot4imp}, restricted to $[0,T^*-1]$. Then by \eqref{boot4imp} we have $\|\l v\r^M f(T^*-1)\|_{L^\infty_{xv}}<A_0\varepsilon_0<\bar{\varepsilon}_1<\bar{\varepsilon}_0$. Applying Theorem \ref{local existence} to \eqref{KWE} with initial data $f(T^*-1),$
we may extend $f$ up to $T^*+2$. Moreover, by \eqref{dispersive estimate local} combined with the fact that \eqref{boot4imp} holds up to time $T^*-1$, we obtain
\begin{equation}\label{pre bootstrap}
 \sup_{t\in [0;T^*+2]}\|\l v\r^{M}f(t)\|_{L^\infty_{xv}}<A_0^2\varepsilon_0.   
\end{equation}
By Proposition \ref{boot main prop}, bound \eqref{pre bootstrap} implies
\eqref{boot4imp} up to time $T^*+2$. Similarly we can show that \eqref{boot1imp}-\eqref{boot3imp} hold up to time $T^*+2$ as well. However, this contradicts the maximality of $T^*$, thus $T^*=+\infty$.

Hence,  \eqref{KWE} with initial data $f_0$ has a solution $f$ in $[0,+\infty)$, which satisfies the dispersive bound
 \begin{align} \label{global bounds positive}
        \sup_{t\ge 0} \big\{ \langle t \rangle^{3/2} \Vert \langle v \rangle f(t) \Vert_{L^{2}_x L^1_v} + \langle t \rangle^{3/2} \Vert \langle v \rangle^\alpha f(t) \Vert_{L^{\infty}_x L^2_v} + \Vert f(t) \Vert_{L^{1}_{xv}} + \Vert \langle v \rangle^M f(t) \Vert_{L^{\infty}_{xv}} \big\}<A\varepsilon_0,
    \end{align}
    where $A=4A_0$.
This solution is also unique by Theorem \ref{local existence}.

Let us now briefly explain how to extend the solution to  negative times. Consider the initial value problem 
\begin{equation}\label{negative equation}
\begin{cases}
 \partial_t g -v\cdot\nabla_x g=-\mathcal{C}[g], \quad t>0\\
 g(t=0)=f_0
\end{cases}
\end{equation}
Since our argument for solving \eqref{KWE} was perturbative, we may apply an identical reasoning for \eqref{negative equation} instead to obtain a unique solution $g\in \mathcal{C}([0;+\infty),\l v\r^{-M}L^\infty_{xv})$ which satisfies the dispersive bound 
$$\sup_{t\ge 0} \big\{ \langle t \rangle^{3/2} \Vert \langle v \rangle g(t) \Vert_{L^{2}_x L^1_v} + \langle t \rangle^{3/2} \Vert \langle v \rangle^\alpha g(t) \Vert_{L^{\infty}_x L^2_v} + \Vert g(t) \Vert_{L^{1}_{xv}} + \Vert \langle v \rangle^M g(t) \Vert_{L^{\infty}_{xv}} \big\}<A\varepsilon_0.$$
We can then define
$f(-t):=g(t)$ for $t>0$. Then, one can easily see that $f\in \mathcal{C}(\R,\l v\r^{-M}L^\infty_{xv})$ is the unique global solution of \eqref{KWE} with initial data $f_0$, and that it satisfies \eqref{global bounds}.

Finally, we show continuity with respect to initial data in the $L^1_{xv}$ topology. Assume $f,g$ are the global solution of \eqref{KWE} with initial data $f_0,g_0\in B_{X_{M,\alpha}}(\varepsilon_0)$ respectively.

By Duhamel's formula and \eqref{conslinweight}, for any $t\in\R$, we have 
\begin{equation}\label{differencefg}
\begin{aligned} 
    &\|f(t) - g(t)\|_{L^1_{xv}}
    \leq \| f_0 - g_0 \|_{L^1_{xv}}\\
    &+ \sum_{\mathcal{T}\in\{\mathcal{L}_1,\mathcal{L}_2,\mathcal{G}_1,\mathcal{G}_2\}}\int_{-|t|}^{|t|} \big( \|\mathcal{T}[f-g,f,f](s)\|_{L^1_{xv}}
    + \|\mathcal{T}[g,f-g,f](s)\|_{L^1_{xv}} + \|\mathcal{T}[g,g,f-g](s)\|_{L^1_{xv}} \big) \, ds.
\end{aligned}
\end{equation}
Now using \eqref{contraction estimate L1} for the nonlinear part followed by \eqref{global bounds}, we find that for some numerical constant $C>0$ 
\begin{equation}\label{continuous dependence estimate}
\begin{aligned}
   \Vert  f(t) - g(t) \Vert_{L^{1}_{xv}} & \leq \Vert  f_0 - g_0 \Vert_{L^1_{xv}} +  C \sup_{s \in \R}  \Vert f(s)-g(s) \Vert_{L^{1}_{xv}}  \\
   & \times \sum_{h_1,h_2 \in \lbrace f,g \rbrace}   \sup_{s \in \R} \Vert \langle v \rangle^3 h_1 \Vert_{L^{\infty}_{xv}} \cdot  \sup_{s \in \R} \langle s \rangle^{3/2} \Vert \langle v \rangle h_2 \Vert_{L^{\infty}_x L^{1}_{v}} \int_{-\infty}^{+\infty} \l s\r^{-3/2}\,ds \\
   &   \leq \Vert  f_0 - g_0  \Vert_{L^1_{xv}} + 12 A^2 C \varepsilon_0^2 \sup_{s \in \R} \Vert f(s) - g(s) \Vert_{L^{1}_{xv}}\\
   &\leq  \Vert  f_0 - g_0  \Vert_{L^1_{xv}} + \frac{1}{2} \sup_{s \in \R} \Vert f(s) - g(s) \Vert_{L^{1}_{xv}},
\end{aligned}
\end{equation}
where to obtain the last line we used the fact that $\varepsilon_0<\bar{\varepsilon}$ and the fact that $\bar{\varepsilon}$ is sufficiently small. Estimate \eqref{contindep} clearly follows. 

\end{proof}

\subsection{Propagation of $L^1_{xv}$ -- moments and conservation laws}
In this section we prove Corollary \ref{propagation moments} which shows that the global solution to \eqref{KWE} propagates $L^1_{xv}$ -- moments as long as they are initially small enough. Moreover, using the propagation of $L^1_{xv}$ -- moments, we show that the solution  conserves the total mass, momentum, and energy.

\begin{proof}[Proof of Corollary \ref{propagation moments}]
 We first prove the propagation of $L^1_{xv}$ -- moments \eqref{propagation of moments statement}. Consider $0<\bar{\varepsilon}_2<\bar{\varepsilon}$ sufficiently small. Let $0<\varepsilon_0<\bar{\varepsilon}_2$, $f_0\in B_{X_{M,\alpha}}(\varepsilon_0)$ and let $f$ denote the global solution of \eqref{KWE} obtained in Theorem \ref{thmdispersive}.
 
Fix $t\in \R$. By Duhamel's formula and \eqref{conslinweight} we find 
\begin{align}\label{Duhamel moments}
    \Vert \langle v \rangle^N f(t) \Vert_{L^1_{xv}} \leq \Vert \langle v \rangle^N f_0 \Vert_{L^1_{xv}} + \sum_{\mathcal{T} \in \lbrace \mathcal{L}_1, \mathcal{L}_2, \mathcal{G}_1, \mathcal{G}_2 \rbrace} \int_{-|t|}^{|t|}  \Vert \langle v \rangle^N \mathcal{T}[f,f,f](s) \Vert_{L^1_{xv}}  \, ds.
\end{align}
Using \eqref{joint L1 lw}, for $-|t|\leq s\leq |t|$ we have
\begin{align} \label{moments collision upper bound}
\begin{split}
    \Vert \langle v \rangle^N \mathcal{T}[f,f,f](s) \Vert_{L^1_{xv}} & \lesssim \Vert \langle v \rangle^N f \Vert_{L^1_{xv}} \Vert \langle v \rangle f(s) \Vert_{L^\infty_x L^1_v} \Vert \langle v \rangle^3 f \Vert_{L^\infty_{xv}} \\
    & \lesssim A^2\varepsilon_0^2 \langle s \rangle^{-3/2} \sup_{s \in\R} \Vert \langle v \rangle^N f(s) \Vert_{L^1_{xv}},
\end{split}
\end{align}
where for the last line we used the global bounds \eqref{global bounds}. Plugging that estimate into \eqref{Duhamel moments}, we obtain for some constant $C>0$
\begin{align} \label{a priori moments}
\begin{split}
     \Vert \langle v \rangle^N f(t) \Vert_{L^1_{xv}} &\leq \varepsilon_0 + C A^2\varepsilon_0^2 \cdot  \sup_{s \in \R} \Vert \langle v \rangle^N f(s) \Vert_{L^1_{xv}}\,\int_{-\infty}^{+\infty} \l s\r^{-3/2}\,ds\\
     &<\varepsilon_0 +\frac{1}{2} \sup_{s \in \R} \Vert \langle v \rangle^N f(s) \Vert_{L^1_{xv}},
    \end{split}
\end{align}
since $0<\varepsilon_0<\bar{\varepsilon}_2$ and $\bar{\varepsilon}_2$ is sufficiently small. Since $t$ was arbitrary, the above estimate implies \eqref{propagation of moments statement}.

Now we prove the conservation laws \eqref{conservation laws}. 

First, we note that for any $\mathcal{T}\in\{\mathcal{L}_1,\mathcal{L}_2,\mathcal{G}_1,\mathcal{G}_2\}$ and $s\in\R$, \eqref{joint L1 lw}, \eqref{propagation moments} and \eqref{global bounds} imply
\begin{align*}
\|\l v\r^2\mathcal{T}[f,f,f](s)\|_{L^1_{xv}}&\lesssim\|\l v\r f(s)\|_{L^\infty_x L^1_{v}}\|\l v\r^3 f\|_{L_{xv}^\infty}\|\l v\r^2 f\|_{L^1_{xv}}\\
&\leq \|\l v\r^{-3^-}\|_{L^1_v}\|\l v\r^{4^+}f\|_{L^\infty_{xv}}\|\l v\r^3 f \|_{L^\infty_{xv}}\|\l v\r^2 f\|_{L^1_{xv}}\\
&\lesssim \varepsilon_0^3.
\end{align*}
 Therefore $\mathcal{L}[f], \mathcal{G}[f]\in L^\infty(\R, \l v\r^{-2} L^1_{xv})$, hence $\mathcal{C}[f]= \mathcal{G}[f]- \mathcal{L}[f]\in L^\infty(\R, \l v\r^{-2} L^1_{xv})$ as well.

Now fix $\phi\in\{1,v,|v|^2\}$ and $t\in  \R$. Relabeling variables, we obtain
\begin{align*}
\int_{\R^3\times\R^3}\phi\,\mathcal{G}[f]\,dv\,dx&=  \frac{1}{2}\int_{\R^{15}}\delta(\Sigma)\delta(\Omega)(\phi+\phi_1)f_1f_2f_3\,dv\,dv_1\,dv_2\,dv_3\,dx,\\
\int_{\R^3\times\R^3}\phi\,\mathcal{L}[f]\,dv\,dx&=  \frac{1}{2}\int_{\R^{15}}\delta(\Sigma)\delta(\Omega)(\phi_2+\phi_3)f_1f_2f_3\,dv\,dv_1\,dv_2\,dv_3\,dx.
\end{align*}

Hence 
\begin{equation}\label{collisional averaging}
    \int_{\R^3\times\R^3}\phi\,\mathcal{C}[f]\,dv\,dx=\frac{1}{2}\int_{\R^{15}}\delta(\Sigma)\delta(\Omega)(\phi+\phi_1-\phi_2-\phi_3)f_1f_2f_3\,dv\,dv_1\,dv_2\,dv_3\,dx=0,
\end{equation}
since the integrand vanishes due to the resonant conditions. 

Now, by Duhamel's formula, Fubini's theorem and \eqref{collisional averaging}  we obtain
\begin{align}
\int_{\R^{6}}\phi(v)f(t,x,v)\,dv\,dx
&=\int_{\R^{6}}\phi(v)f_0(x-tv,v)\,dv\,dx+\int_0^t\int_{\R^3}\int_{\R^3} \phi(v)\mathcal{C}[f](s,x+(s-t)v,v)\,dv\,dx\,ds\nonumber\\
&=\int_{\R^{6}}\phi(v)f_0(x,v)\,dv\,dx+\int_0^t\int_{\R^3}\int_{\R^3} \phi(v)\mathcal{C}[f](s,x,v)\,dv\,dx\,ds\nonumber\\
&=\int_{\R^{6}}\phi(v)f_0(x,v)\,dv\,dx\nonumber,
\end{align}
and the result is proved.

\end{proof}

\section{Scattering} \label{section:scattering}

In this section we study the asymptotic behavior of solutions constructed in Section \ref{section:GWP}. We prove that they scatter, meaning that they behave like free transport as time reaches infinity. We then refine our understanding of the scattering states by studying properties of wave operators, thereby proving an asymptotic completeness result for \eqref{KWE}. The upshot of our analysis is that the scattering state characterizes the evolution fully, and moreover that any reasonably fast decaying distribution near vacuum can be achieved as a scattering state of \eqref{KWE}.

The results are mostly obtained as a byproduct of the dispersive analysis carried out in Section \ref{section:GWP}. As a result many of the arguments are similar, therefore we will omit some details in the proofs.

\subsection{Scattering states}
We start by constructing the scattering states. 
Recall the definition of the set $B_{X_{M,\alpha}}(\varepsilon_0)$ given in \eqref{X ball}. 

\begin{proposition} \label{scatteringstates}
Let $f_0 \in B_{X_{M,\alpha}}(\varepsilon_0),$ where $0<\varepsilon_0<\overline{\varepsilon}$ (given in Theorem \ref{thmdispersive}).

Then there exist two unique functions $f_{-\infty}, f_{+\infty} \in L^1_{xv} \cap \langle v \rangle^{-M} L^{\infty}_{xv}$ such that if $f(t)$ denotes the global solution to \eqref{KWE} with initial data $f_0$, then
\begin{align} \label{scatt+infty}
\begin{split}
 &   \Vert \langle v \rangle^M \big( f(t) - \mathcal{S}(t) f_{+\infty} \big) \Vert_{L^\infty_{xv}} +  \Vert f(t) - \mathcal{S}(t) f_{+\infty} \Vert_{L^1_{xv}}\longrightarrow_{t \to \infty} 0 \\
\end{split}
\end{align}
and 
\begin{align} \label{scatt-infty}
\begin{split}
& \Vert \langle v \rangle^M \big( f(t) - \mathcal{S}(t) f_{-\infty} \big) \Vert_{L^\infty_{xv}} +  \Vert f(t) - \mathcal{S}(t) f_{-\infty} \Vert_{L^1_{xv}} \longrightarrow_{t \to -\infty} 0. \\
\end{split}
\end{align}
\end{proposition}
\begin{remark}
As a byproduct of the proof we obtain the decay rate of $t^{-1/2}$ for the convergence in \eqref{scatt+infty} and \eqref{scatt-infty}. 
\end{remark}
\begin{proof}
Notice that as a direct consequence of \eqref{mainL1} and \eqref{global bounds},
for any $s>0$, we have
\begin{align} 
\|\mathcal{S}(-s) \mathcal{C}[f](s)\|_{L^1_{xv}}&\lesssim \l s\r^{-3/2} \sup_{s \in \mathbb{R}}\left\{ \Vert f \Vert_{L^1_{xv}} \cdot \Vert \l v \r^M f \Vert_{L^\infty_{xv}} \cdot
\langle s \rangle^{3/2} \Vert \l v \r^\alpha f(s) \Vert_{L^{\infty}_x L^2_v} \right\} \nonumber\\
&\leq A^3\varepsilon_0^3\l s\r^{-3/2} \label{pointwise bound scattering},
\end{align}
hence, for any $t>0$, we have
\begin{align} \label{integral valid in L1}
\begin{split}
   \int_{0}^{t} \|\mathcal{S}(-s) \mathcal{C}[f](s)\|_{L^1_{x,v}} \, ds  \lesssim A^3\varepsilon_0^3\int_{0}^{\infty} \l s \r^{-3/2}\,ds <\infty,
\end{split}
\end{align}
thus $\int_0^t\mathcal{S}(-s)\mathcal{C}[f](s)\,ds$ is well-defined in $L^1_{xv}$ for all $t>0$.

Now, considering $1<t_1<t_2$, and integrating \eqref{pointwise bound scattering}, we have
\begin{align} \label{decayscattering}
\begin{split}
   \int_{t_1}^{t_2} \|\mathcal{S}(-s) \mathcal{C}[f](s)\|_{L^1_{xv}} \, ds  \lesssim A^3\varepsilon_0^3 \int_{t_1}^{t_2} \frac{ds}{\l s \r^{3/2}} \leq A^3\varepsilon_0^3 \,t_1^{-1/2} \longrightarrow_{t_1, t_2 \to + \infty} 0 .
\end{split}
\end{align}
Therefore $\big( \int_{0}^{t} \mathcal{S}(-s) \mathcal{C}[f] \, ds \big)_t$ is Cauchy in $L^{1}_{xv}$, hence its limit as $t \to +\infty$ belongs to $L^1_{xv}$. Let us define
\begin{align} \label{definition scattering state +infty}
f_{+\infty} := f_0 + \int_{0}^{\infty} \mathcal{S}(-s) \mathcal{C}[f](s) \, ds\in L^1_{xv}.
\end{align}
Then using Duhamel's formula, \eqref{conslinweight}, and \eqref{pointwise bound scattering}, we obtain 
\begin{align*}
\|f(t)-\mathcal{S}(t)f_{+\infty}\|_{L^1_{xv}}&\leq \int_t^\infty\|\mathcal{S}(t-s)\mathcal{C}[f](s)\|_{L^1_{xv}}\,ds
\lesssim A^3\varepsilon_0^3\, t^{-1/2}\longrightarrow_{t \to + \infty} 0.
\end{align*}

Similarly, using \eqref{mainLinfty} instead of \eqref{mainL1}, we obtain that  $f_{+\infty}\in\langle v \rangle^{-M} L^{\infty}_{xv}$, and that $\Vert \langle v \rangle^M \big( f(t) - \mathcal{S}(t) f_{+\infty} \big) \Vert_{L^\infty_{xv}}\longrightarrow_{t\to+\infty} 0$.

Uniqueness of the scattering state $f_{+\infty}$ follows by uniqueness of the limit in either $L^1_{xv}$ or $\l v\r^{-M}L^\infty_{xv}$.

In the exact same way we can show that there exists a unique scattering state as $t\to -\infty$  given by 
\begin{align} \label{definition scattering state -infty}
    f_{-\infty} := f_0 - \int_{-\infty}^{0} \mathcal{S}(-s) \mathcal{C}[f] \, ds
\end{align}
that satisfies \eqref{scatt-infty}.
\end{proof}

As a consequence of Proposition \ref{scatteringstates} (in particular the uniqueness of the scattering states), for $0<\bar{\varepsilon}_0<\bar{\varepsilon}$, we can define the mappings
\begin{align} \label{defU+}
    U_{+} : \begin{cases} 
 B_{X_{M,\alpha}}(\varepsilon_0) & \longrightarrow U_{+} \big( B_{X_{M,\alpha}}(\varepsilon_0) \big) \\
 f_0 & \longmapsto f_{+\infty}
    \end{cases}
\end{align}
and 
\begin{align} \label{defU-}
    U_{-} : \begin{cases} 
 B_{X_{M,\alpha}}(\varepsilon_0) & \longrightarrow U_{-} \big( B_{X_{M,\alpha}}(\varepsilon_0) \big) \\
 f_0 & \longmapsto f_{-\infty}
    \end{cases} .
\end{align}

\subsection{Continuity of $U_{+}, U_{-}$} Here, we prove that the scattering maps are continuous.

\begin{proposition} \label{scatteringcontinuity}
Let $0<\varepsilon_0<\bar{\varepsilon}$, where $\bar{\varepsilon}$ is given in  the statement of Theorem \ref{thmdispersive}.   Then the operators
\begin{align*} 
    U_{\pm} : \begin{cases} 
\big( B_{X_{M,\alpha}}(\varepsilon_0), \Vert \cdot \Vert_{L^1_{xv}} \big) & \longrightarrow \big( U_{\pm } \big(B_{X_{M,\alpha}}(\varepsilon_0) \big), \Vert \cdot \Vert_{L^1_{xv}} \big) \\
 f_0 & \longmapsto f_{\pm \infty}
    \end{cases}
\end{align*}
are Lipschitz continuous. More precisely,
\begin{align*}
    \Vert U_{\pm} f_0 - U_{\pm} g_0 \Vert_{L^1_{xv}} \leq 2 \Vert f_0 - g_0 \Vert_{L^1_{xv}},\quad\forall\, f_0,g_0\in B_{X_{M,\alpha}}(\varepsilon_0).
\end{align*}
\end{proposition}
\begin{proof}
We carry out the proof for $U_{+},$ the other case being similar.

Let $0<\varepsilon_0<\bar{\varepsilon}$, $f_0,g_0\in B_{X_{M,\alpha}}(\varepsilon_0)$ and  $f,g$ be the solutions constructed using Theorem \ref{thmdispersive} with initial data $f_0$ and $g_0$ respectively. 


This implies that for any $t>0$
\begin{align*}
 \Vert U_{+} f_0 - U_{+} g_0 \Vert_{L^1_{xv}} & =   \Vert \mathcal{S}(t) U_{+} f_0 - \mathcal{S}(t) U_{+} g_0 \Vert_{L^1_{xv}}   \\
 & \leq \Vert f(t) - g(t) \Vert_{L^1_{xv}} + \Vert f(t) - \mathcal{S}(t) U_{+} f_0 \Vert_{L^1_{xv}} + \Vert g(t) - \mathcal{S}(t) U_{+} g_0 \Vert_{L^1_{xv}} \\
 & \leq 2 \Vert  f_0 - g_0 \Vert_{L^1_{xv}} +  \Vert f(t) - \mathcal{S}(t) U_{+} f_0 \Vert_{L^1_{xv}} + \Vert g(t) - \mathcal{S}(t) U_{+} g_0 \Vert_{L^1_{xv}},
\end{align*}
where to obtain the last line we used \eqref{contindep}. Letting $t \to +\infty$ and using \eqref{scatt+infty}, the result follows.

\end{proof}


\subsection{Injectivity of $U_{+},U_{-}$}
Now,   we prove that the scattering maps are injective, i.e. different initial data scatter to different states.

\begin{proposition} \label{scatteringinjectivity}
    Let $\bar{\varepsilon},A$ be as in the statement of Theorem \ref{thmdispersive}. Then, there exists $0<\bar{\varepsilon}_3<\bar{\varepsilon}$  such that for $0<\varepsilon_0<\bar{\varepsilon}_3$, the operators $U_{\pm} : B_{X_{M,\alpha}}(\varepsilon_0) \rightarrow U_{\pm}\big( B_{X_{M,\alpha}}(\varepsilon_0) \big)$ are injective.
\end{proposition}
\begin{proof} 
 Consider $0<\bar{\varepsilon}_3<\bar{\varepsilon}$ sufficiently small, and $0<\varepsilon_0<\bar{\varepsilon}_3$. 
Let $f_0,g_0 \in B_{X_{M}}(\varepsilon_0)$ be such that $U_{+} f_0 = U_{+} g_0.$ Then by definition of $U_+$ \eqref{definition scattering state +infty}, we have
$$\vert f_0-g_0 \vert = \bigg \vert \int_0^t\mathcal{S}(-s)(\mathcal{C}[g](s)-\mathcal{C}[f](s))\,ds \bigg \vert.$$
Hence, using \eqref{conslinweight} followed by \eqref{contraction estimate L1}, \eqref{global bounds} and \eqref{contindep}  we obtain
\begin{align*}
&\|f_0-g_0\|_{L^1_{xv}}\leq\int_0^{+\infty}\|\mathcal{C}[f](s)-\mathcal{C}[g](s)\|_{L^1_{xv}}\,ds\\
&\leq  \sum_{\mathcal{T}\in\{\mathcal{L}_1,\mathcal{L}_2,\mathcal{G}_1,\mathcal{G}_2\}}\int_{0}^{+\infty} \big( \|\mathcal{T}[f-g,f,f](s)\|_{L^1_{xv}} + \|\mathcal{T}[g,f-g,f](s)\|_{L^1_{xv}} + \|\mathcal{T}[g,g,f-g](s)\|_{L^1_{xv}} \big) \, ds\\
&\lesssim  \sup_{s \in [0;t]}  \Vert f(s)-g(s) \Vert_{L^{1}_{xv}} \sum_{h_1,h_2 \in \lbrace f,g \rbrace}   \sup_{s \in \R} \Vert \langle v \rangle^3 h_1 \Vert_{L^{\infty}_{xv}} \cdot  \sup_{s \in \R} \langle s \rangle^{3/2} \Vert \langle v \rangle h_2 \Vert_{L^{\infty}_x L^{1}_{v}} \int_{0}^{+\infty} \l s\r^{-3/2}\,ds\\
&\lesssim 2A^2\varepsilon_0^2 \|f_0-g_0\|_{L^1_{xv}}<\frac{1}{2} \|f_0-g_0\|_{L^1_{xv}},
\end{align*}
since $\varepsilon_0<\bar{\varepsilon}_3$, and $\bar{\varepsilon}_3$ is sufficiently small. We conclude that $f_0=g_0$, thus $U_+$ is injective. 

The reasoning for the $U_{-}$ is similar, thus we omit the details.



\end{proof}

\subsection{Surjectivity of $U_{+}, U_{-}$} We finally prove that the scattering maps are surjective. In other words any appropriately decaying function can be viewed as a scattering state.
\begin{proposition} \label{scatteringsurjectivity}
Let $\bar{\varepsilon}$ be as in Theorem \ref{thmdispersive}. Then, there exists $0<\overline{\varepsilon}_4<\frac{\bar{\varepsilon}}{2}$ such that for any $0<\varepsilon_0<\bar{\varepsilon}_4$ and any $f_{+\infty} \in B_{X_{M,\alpha}}(\varepsilon_0)$,  there exists unique $f_0 \in B_{X_{M,\alpha}}(2\varepsilon_0)$  such that the global solution $f$ to \eqref{KWE} with initial data $f_0$ satisfies 
\begin{align}\label{surjective convergence}
    \Vert f(t) - \mathcal{S}(t) f_{+\infty}\Vert_{L^1_{xv}} +  \big \Vert \l v \r^M \big( f(t) - \mathcal{S}(t) f_{+\infty} \big) \big \Vert_{L^\infty_{xv}} \longrightarrow_{t \to  + \infty} 0.
\end{align}
\end{proposition}
\begin{remark}
A similar statement holds for $t \to -\infty.$
\end{remark}
 \begin{proof} 


Consider $0<\bar{\varepsilon}_4<\frac{\bar{\varepsilon}}{2}$ sufficiently small and let $0<\varepsilon_0<\bar{\varepsilon}_4$.  Given $f_{+\infty}\in B_{X_{M,\alpha}}(\varepsilon_0) $
we define the map 
\begin{align} \label{defPhiinfty}
\Phi:f \mapsto \mathcal{S}(t) f_{+\infty} - \int_t^{+\infty} \mathcal{S}(t-s) \mathcal{C}[f] \, ds.
\end{align}
To prove the result, it suffices to show that $\Phi$ is a contraction on the ball  
$$\bm{\widetilde{B}}: =\Big\{f\in\mathcal{C}([0;+\infty),X_{M,\alpha})\,:\,\sup_{t\ge 0}\|f(t)\|_{X_{M,\alpha}}\leq 2\varepsilon_0\Big\}.$$
Indeed, if that is the case, the contraction mapping principle implies that $\Phi$ has a unique fixed point $f\in\bm{\widetilde{B}}$ i.e.
\begin{equation}\label{fixed point}
f(t)=\mathcal{S}(t)f_{+\infty}-\int_t^{+\infty}\mathcal{S}(t-s)\mathcal{C}[f](s)\,ds,\quad\forall t\ge 0.    
\end{equation} 
Defining $f_0=f(0)$, we take that $f_0\in B_{X_M}(2\varepsilon_0)$. Then filtering \eqref{fixed point} by $\mathcal{S}(-t)$, we take that $f_{+\infty}=U_+ f_0$ which is equivalent to \eqref{surjective convergence}.

Hence, we now focus on proving that $\Phi$ is a contraction of $\bm{\widetilde{B}}$.
\newline
\newline
\underline{Stability:} \newline
Let $f\in\bm{\widetilde{B}}$. Given $\mathcal{T}\in\{\mathcal{L}_1,\mathcal{L}_2,\mathcal{G}_1,\mathcal{G}_2\}$,    let us denote
\begin{align*}
    J_{\mathcal{T}}(t) := \int_t^{+\infty} \mathcal{S}(t-s) \mathcal{T}[f,f,f] \, ds.
\end{align*} 
Since $f_{+\infty}\in B_{X_{M,\alpha}}(\varepsilon_0) $, triangle inequality implies
\begin{equation}\label{triangle bound on Phi}
\|\Phi f(t)\|_{X_{M,\alpha}}\leq\varepsilon_0+\sum_{\mathcal{T}\in\{\mathcal{L}_1,\mathcal{L}_2,\mathcal{G}_1,\mathcal{G}_2\}} \|J_{\mathcal{T}}(t)\|_{X_{M,\alpha}}.   
\end{equation}

Fix $\mathcal{T}\in\{\mathcal{L}_1,\mathcal{L}_2,\mathcal{G}_1,\mathcal{G}_2\}$ and $t\ge 0$.
Using \eqref{joint L1} and the Cauchy-Schwarz inequality in $v$ we obtain
\begin{align}
    \Vert J_{\mathcal{T}}(t) \Vert_{L^1_{xv}} & \lesssim \int_{t}^{+\infty} \Vert \mathcal{T}[f,f,f] \Vert_{L^1_{xv}} \, ds\nonumber \\
    & \lesssim \sup_{s\ge t} \langle s \rangle^{3/2} \Vert \langle v \rangle^{\alpha} f(s) \Vert_{L^\infty_x L^2_v} \cdot \sup_{s \in [t;+\infty)} \Vert f(s) \Vert_{L^1_{xv}} \nonumber\\
    & \hspace{1cm}\times \sup_{s\ge t} \Vert \langle v \rangle f(s) \Vert_{L^\infty_{xv}} \int_{t}^{+\infty} \langle s \rangle^{-3/2} \, ds\nonumber \\
    &\lesssim \varepsilon_0^3<\frac{\varepsilon_0}{4},\label{Phi L1 bound}
\end{align}
since $\varepsilon_0<\bar{\varepsilon}_4$, and $\bar{\varepsilon}_4$ is sufficiently small.

Similarly,  using \eqref{joint weighted Linfty} instead, we obtain 
\begin{equation}\label{Phi Linf bound}
    \Vert \langle v \rangle^M J_{\mathcal{T}}(t) \Vert_{L^\infty_{xv}} \lesssim\varepsilon_0^3<\frac{\varepsilon_0}{4},
\end{equation}
since $\varepsilon_0<\bar{\varepsilon}_4$, and $\bar{\varepsilon}_4$ is sufficiently small.

For $\| \langle v \rangle^\alpha J_{\mathcal{T}}(t)\|_{L^\infty_x L^2_v}$, we decompose 
\begin{align}\label{triangle inequality on J}
  \| \langle v \rangle^\alpha J_{\mathcal{T}}(t)\|_{L^\infty_x L^2_v}\leq \| \langle v \rangle^\alpha J_{\mathcal{T}}(t+1)\|_{L^\infty_x L^2_v}+ \big \| \langle v \rangle^\alpha \big( J_{\mathcal{T}}(t)-J_{\mathcal{T}}(t+1) \big) \big\|_{L^\infty_x L^2_v}  .
\end{align}
To estimate the bulk term, we use \eqref{commutation identity}-\eqref{decaylin}, as well as  \eqref{joint weighted Linfty} to obtain 
\begin{align*}
    \Vert \langle v \rangle^\alpha J_{\mathcal{T}}(t+1) \Vert_{L^\infty_x L^2_v} & \leq \int_{t+1}^{+\infty} \frac{1}{(s-t)^{3/2}} \Vert \langle v \rangle^\alpha \mathcal{T}[f,f,f] \Vert_{L^2_x L^\infty_v} \, ds \\
    &\lesssim \sup_{s\ge t+1} \langle s \rangle^{3/2} \Vert \langle v \rangle f(s) \Vert_{L^2_x L^1_v} \cdot \big(\sup_{s\geq t+1}  \Vert \langle v \rangle^3 f(s) \Vert_{L^\infty_{xv}} \big)^2  \\
    & \times \int_{t+1}^{+\infty} \frac{ds}{(s-t)^{3/2} \langle s \rangle^{3/2}}\\
    &\lesssim \l t\r^{-3/2}\varepsilon_0^3.
\end{align*}
We may also bound in the same way
\begin{align*}
\big \|  \langle v \rangle^\alpha \big( J_{\mathcal{T}}(t)-J_{\mathcal{T}}(t+1) \big)  \big \|_{L^\infty_x L^2_v}& \lesssim \varepsilon_0^3\int_{t}^{t+1} \frac{ds}{\langle s \rangle^{3/2}}\lesssim \l t\r^{-3/2}\varepsilon_0^3.   
\end{align*}
By \eqref{triangle inequality on J} we obtain 
\begin{equation}\label{Phi LinfL2}
\| \langle v \rangle^\alpha J_{\mathcal{T}}(t)\|_{L^\infty_x L^2_v}\lesssim\l t\r^{-3/2}\varepsilon_0^3<\l t\r^{-3/2}\frac{\varepsilon_0}{4},
\end{equation}
since $\varepsilon_0<\bar{\varepsilon}_4$ and $\bar{\varepsilon}_4$ is sufficiently small.

Finally, by a similar argument using \eqref{joint L2 1w} instead, obtain 
\begin{equation}\label{Phi L2L1 bound}
\|  \langle v \rangle J_{\mathcal{T}}(t)\|_{L^2_x L^1_v}\lesssim\l t\r^{-3/2}\varepsilon_0^3<\l t\r^{-3/2}\frac{\varepsilon_0}{4},
\end{equation}
since $\varepsilon_0<\bar{\varepsilon}_4$ and $\bar{\varepsilon}_4$ is sufficiently small.

Combining \eqref{triangle bound on Phi}, \eqref{Phi L1 bound}, \eqref{Phi Linf bound}, \eqref{Phi LinfL2}, \eqref{Phi L2L1 bound}, we obtain
$$\|\Phi f\|_{\mathcal{C}([0;+\infty),X_{M,\alpha})}<2\varepsilon_0.$$
Moreover $\Phi f$ is clearly continuous in time, hence $\Phi:\bm{B}\to\bm{B}$.
\newline
\newline
\underline{Contraction:} \newline
The proof is almost identical to stability, relying this time on Lemma \ref{contractscatt} instead, therefore we omit the details of the proof. For $f,g \in \bm{\widetilde{B}}$, we obtain
\begin{align*}
    \|\Phi f -\Phi g\|_{\mathcal{C}([0;+\infty),X_{M,\alpha})}\lesssim\varepsilon_0^2 \| f- g\|_{\mathcal{C}([0;+\infty),X_{M,\alpha})}\leq\frac{1}{2} \| f- g\|_{\mathcal{C}([0;+\infty),X_{M,\alpha})},
\end{align*}
since $\varepsilon_0<\bar{\varepsilon}_4$, and $\bar{\varepsilon}_4$ is sufficiently small.

The result is proved.
\end{proof}

\section{Positivity of the solution} \label{sec:non-negative solution} 
In this section we show that \eqref{KWE} preserves positivity forward in time i.e. for positive initial data, the solution obtained in Theorem \ref{thmdispersive} remains positive for $t\geq 0$. This comes in agreement with the fact that the solution represents an energy spectrum. In order to show positivity, we will rely on the monotonicity properties of the equation. Namely, we will employ a classical tool from the kinetic theory of particles, the so-called Kaniel-Shinbrot \cite{KS} iteration scheme. The idea is to approximate the solution of the equation with positive subsolutions and supersolutions, and pass to the limit to retrieve the global solution of the nonlinear equation \eqref{KWE}. Unlike classical applications of the Kaniel-Shinbrot \cite{KS} iteration, we will need to show that the iterates satify dispersive bounds as well. This can be achieved by initializing the scheme properly, namely by solving the gain only equation \eqref{KWE}.

The main result of this section is the following:

\begin{theorem} \label{thmpos}
   Let $\overline{\varepsilon}, A$ be as in the statement of Theorem \ref{thmdispersive}. Then, there exists $0<\bar{\varepsilon}_5<\bar{\varepsilon}$ such that if $0<\varepsilon_0<\overline{\varepsilon}_5$ and $f_0 \in B_{X_{M,\alpha}}(\varepsilon_0)$ with $f_0\geq 0$, the global solution in Theorem \ref{thmdispersive} is non-negative for all $t\ge 0$. 
\end{theorem}

\subsection{The gain only equation}
We start by studying the gain only equation, since that will allow us to properly initialize our iteration scheme. We will show that the gain only equation has a unique solution forward in time which is dispersive, and that for positive initial data, the corresponding solution remains positive. First we give a definition of a solution to the gain only equation \eqref{KWE}:
\begin{definition}\label{definition of a solution gain only} Let $I\subseteq\R$ be an interval with $0\in I$, and $f_0\in X_{M,\alpha}$. We say that $f\in \mathcal{C}(I,\l v\r^{-M}L^\infty_{x,v})$ is a solution of the gain only \eqref{KWE} in $I$ with initial data $f_0$, if
\begin{equation}\label{Duhamel formula gain}
f(t)=\mathcal{S}(t)f_0+\int_0^t\mathcal{S}(t-s)\mathcal{G}[f](s)\,ds,\quad\forall t\in I.    
\end{equation}  
\end{definition}
\begin{remark}
Analogous conclusions to Remark \ref{remark on well defined terms} hold for $\int_0^t\mathcal{S}(t-s)\mathcal{G}[f](s)\,ds$.
\end{remark}

\begin{proposition}\label{proposition gain only}
Let  $\overline{\varepsilon}, A$ be as in the statement of Theorem \ref{thmdispersive}. Then, for $0<\varepsilon_0<\overline{\varepsilon}$ and $f_0 \in B_{X_{M,\alpha}}(\varepsilon_0)$ with $f_0\ge 0$, the gain equation \eqref{KWE}  has a unique non-negative solution $\widetilde{f}\in \mathcal{C}([0;+\infty),\l v\r^{-M}L^\infty_{xv})$ which is dispersive, meaning it satisfies the bound 
 \begin{align}\label{dispersive bound gain only}
        \sup_{t \ge 0} \big\{ \langle t \rangle ^{3/2} \Vert \langle v \rangle \widetilde{f}(t) \Vert_{L^{2}_x L^1_v} + \langle t \rangle ^{3/2} \Vert \langle v \rangle^\alpha \widetilde{f}(t) \Vert_{L^{\infty}_x L^2_v} + \Vert \widetilde{f}(t) \Vert_{L^{1}_{xv}} + \Vert \langle v \rangle^M \widetilde{f}(t) \Vert_{L^{\infty}_{xv}} \big\} < A\varepsilon_0. 
    \end{align}
\end{proposition}
\begin{proof}
Simply set $\mathcal{L}=0$ and all the estimates written in the proof of Theorem \ref{thmdispersive} remain true. Omitting the details, this yields existence and uniqueness of a solution $\widetilde{f}\in \mathcal{C}([0;+\infty),\l v\r^{-M}L^\infty_{xv})$ to the gain only \eqref{KWE}, as well as  the dispersive bound \eqref{dispersive bound gain only}.

For positivity,  we consider the following iteration scheme for $t\geq 0$:
\begin{equation} \label{gain only iteration}
\begin{cases}
  \widetilde{f}_{n+1}=\mathcal{S}(t)f_0+\int_0^t\mathcal{S}(t-s)\mathcal{G}[\widetilde{f}_n](s)\,ds, \quad n\in\mathbb{N},\\
  \widetilde{f}_0=\mathcal{S}(t)f_0.
\end{cases} 
\end{equation}
Since $f_0\geq 0$ and $\mathcal{G}$ is positive, by induction we obtain that $(\widetilde{f}_n)_{n=0}^\infty$ is increasing and non-negative, hence $\widetilde{f}_n\nearrow \widetilde{g}$. Clearly $\widetilde{g}\geq 0$. 


Now, we can use the monotone convergence theorem to pass to the limit in \eqref{gain only iteration}, and obtain that $\widetilde{g}$ satisfies the same initial value problem as $\widetilde{f}$. By uniqueness, we obtain $\widetilde{f} = \widetilde{g},$ hence $\widetilde{f} \geq 0.$
\end{proof}

\subsection{The Kaniel-Shinbrot iteration scheme} We formally first outline the Kaniel-Shinbrot iteration scheme which we will use in order to construct subsolutions and supersolutions to \eqref{KWE KS}.

Using \eqref{collisional operator frequency}, we write the equation as follows: 
\begin{equation}\label{KWE KS}
(\partial_t +v\cdot\nabla_x)f+f\mathcal{R}[f,f]=\mathcal{G}[f],
\end{equation}
where $\mathcal{R},\mathcal{G}$ are given by \eqref{collision frequency} and \eqref{gain operator} respectively.
Let $u_0, l_0:[0,\infty)\times\R_x^3\times\R_v^3\to\R$ with $0\leq l_0\leq u_0$. 
 For $n\in\mathbb{N}$, consider the coupled  linear initial value problems
\begin{equation}\label{lower IVP n}
\begin{cases}
(\partial_t +v\cdot\nabla_x)l_n+ l_n \mathcal{R}[u_{n-1},u_{n-1}]&=\mathcal{G}[l_{n-1}],\\
l_n(0)=f_0,
\end{cases}
\end{equation}

\begin{equation}\label{upper IVP n}
\begin{cases}
(\partial_t +v\cdot\nabla_x)u_n+ u_n \mathcal{R}[l_{n-1},l_{n-1}]&=\mathcal{G}[u_{n-1}],\\
u_n(0)=f_0,
\end{cases}
\end{equation}


Solutions to \eqref{lower IVP n}-\eqref{upper IVP n} are given inductively by
\begin{equation}
\begin{split}\label{solution lower IVP n}
l_n(t)&=(\mathcal{S}(t)f_0)\exp\left( -\int_0^t \mathcal{S}(t-s)\mathcal{R}[u_{n-1},u_{n-1}](s)\,ds\right)\\
&\quad+\int_0^t \mathcal{S}(t-s)\mathcal{G}[l_{n-1}](s)\exp\left(-\int_s^t \mathcal{S}(t-\tau)\mathcal{R}[u_{n-1},u_{n-1}](\tau)\,d\tau \right)\,ds,
\end{split}
\end{equation}
and
\begin{equation}
\begin{split}\label{solution upper IVP n}
u_n(t)&=(\mathcal{S}(t)f_0)\exp\left( -\int_0^t \mathcal{S}(t-s)\mathcal{R}[l_{n-1},l_{n-1}](s)\,ds\right)\\
&\quad+\int_0^t \mathcal{S}(t-s)\mathcal{G}[u_{n-1}](s)\exp\left(-\int_s^t \mathcal{S}(t-\tau)\mathcal{R}[l_{n-1},l_{n-1}](\tau)\,d\tau \right)\,ds.
\end{split}
\end{equation}

\subsubsection*{Initialization}
We first initiate the scheme by choosing $l_0,u_0$ so dispersion is preserved and the following beginning condition 
$$0\leq l_0\leq l_1\leq u_1\leq u_0$$
is satisfied.

\begin{lemma} \label{initialization}
Let $\overline{\varepsilon}, A$ be as in the statement of Theorem \ref{thmdispersive}. Then, for $0<\varepsilon_0<\overline{\varepsilon}$ and $f_0 \in B_{X_{M,\alpha}}(\varepsilon_0)$ with $f_0\geq 0$,
there exist $l_0,u_0$ with $0 \leq l_0 \leq u_0$ for the coupled initial value problems \eqref{lower IVP n}-\eqref{upper IVP n} such that the first iterates $l_1,u_1$ satisfy
\begin{equation}\label{beginning condition}
0\leq l_0\leq l_1\leq u_1\leq u_0.
\end{equation}
Moreover $u_0$ satisfies the dispersive bound \eqref{dispersive bound gain only}.
\end{lemma}
\begin{proof}
Let $\widetilde{f}$ be the solution given by Proposition \ref{proposition gain only} for initial data $f_0\in B_{X_{M,\alpha}}(\varepsilon_0).$

We define $l_0=0$, $u_0=\widetilde{f}$.

Then \eqref{solution lower IVP n}-\eqref{solution upper IVP n} give the following expressions for the first iterates:
\begin{align*}
l_1(t)&=(\mathcal{S}(t)f_0)\exp\left(-\int_0^t \mathcal{S}(t-s)\mathcal{R}[u_0,u_0](s)\,ds\right),\\
u_1(t)&=\mathcal{S}(t)f_0+\int_0^t \mathcal{S}(t-s) \mathcal{G}[u_0](s)\,ds.
\end{align*}
Since the operator $\mathcal{G}, \mathcal{R}$ are positive, and  $\widetilde{f} \geq 0$ by Proposition \ref{proposition gain only}, we have
\begin{align*}
    \int_0^t \mathcal{S}(t-s)\mathcal{R}[u_0,u_0](s)\,ds , \int_0^t \mathcal{S}(t-s) \mathcal{G}[u_0](s)\,ds \geq 0 .
\end{align*}
As a result $0=l_0\leq l_1\leq u_1$.

Note that the initial value problem for $u_1$ is exactly that of the gain-only KWE, therefore by uniqueness in Proposition \ref{proposition gain only} we have $u_1=\widetilde{f}=u_0,$ and we can conclude that the initialization condition is fulfilled. 
\end{proof}

\subsubsection*{Properties of iterates} Now, we prove monotonicity of the scheme.

\begin{proposition}\label{kaniel-shinbrot prop} 
Let $l_n, u_n$ be the solutions to \eqref{solution lower IVP n}-\eqref{solution upper IVP n} for $l_0,u_0$ chosen as in Lemma \ref{initialization}.

Then for all $n\in\mathbb{N}$ and $t\geq 0$, we have
\begin{equation}\label{inductive nesting}
\begin{split}
0&\leq l_0\leq l_1\leq\dots\leq l_{n-1}\leq l_{n}\leq u_n \leq u_{n-1}\leq\dots\leq u_1\leq u_0
\end{split}
\end{equation}
Moreover $(l_n)_{n=0}^\infty$, $(u_n)_{n=0}^\infty$ satisfy the dispersive estimate \eqref{dispersive bound gain only}
\end{proposition}

\begin{proof}
We argue by induction. For $n=1,$ \eqref{inductive nesting} is given by \eqref{beginning condition}, while Proposition \ref{proposition gain only} gives \eqref{global bounds} at the first step. 

Fix $n \in \mathbb{N}.$ Assume that \eqref{inductive nesting} holds for $n$, and that $l_n$ and $u_n$ satisfy \eqref{global bounds}.

Since $l_{n-1} \leq l_n$ and $u_{n} \leq u_{n-1}$,  we have
\begin{equation*}
\mathcal{R}[u_{n},u_{n}]\leq \mathcal{R}[u_{n-1},u_{n-1}], \quad \mathcal{G}[l_{n-1}]\leq \mathcal{G}[l_n],
 \end{equation*}
which, together with \eqref{solution lower IVP n}, implies $l_n\leq l_{n+1}$. 

Similarly,
$u_{n+1}\leq u_n$. 

Finally 
$$\mathcal{R}[l_{n},l_{n}]\leq \mathcal{R}[u_{n},u_{n}],\quad \mathcal{G}[l_{n}]\leq \mathcal{G}[u_{n}],$$
and \eqref{solution lower IVP n}-\eqref{solution upper IVP n} imply $l_{n+1}\leq u_{n+1}$. Hence \eqref{inductive nesting} holds for all $n$ by induction. 

Finally,  the dispersive estimate \eqref{dispersive bound gain only} holds for $(l_n)_{n=0}^\infty$, $(u_n)_{n=0}^\infty$ due to \eqref{inductive nesting}, and the fact that $u_0=\widetilde{f}$.

\end{proof}

\subsubsection*{Conclusion of the proof} 

\begin{proof}[Proof of Theorem \ref{thmpos}]

Consider $0<\bar{\varepsilon}_6<\bar{\varepsilon}$ sufficiently small.
Let $0<\varepsilon_0<\bar{\varepsilon}_5$, $f_0\in B_{X_{M,\alpha}}(\varepsilon_0)$, and consider the sequences $(l_n)_{n=1}^\infty$, $(u_n)_{n=1}^\infty$ constructed in Lemma \ref{kaniel-shinbrot prop}.

For fixed $t\geq 0,$ the sequence $(l_n(t))_{n=0}^\infty$ is increasing and upper bounded by \eqref{inductive nesting}, therefore it converges $l_n(t)\nearrow l(t)$. Similarly the sequence  $(u_n(t))_{n=0}^\infty$ is decreasing and lower bounded so $u_n(t)\searrow u(t)$. Moreover, by \eqref{inductive nesting} we have
$$0\leq l(t)\leq u(t)\leq u_0.$$

Since $u_0 = \widetilde{f}$ satisfies \eqref{dispersive bound gain only}, so do $l$ and $u.$ 

 Integrating \eqref{lower IVP n}-\eqref{upper IVP n} in time and using the dominated convergence theorem to let $n\to\infty$ , we obtain
\begin{align}
l(t)+ \int_0^t \mathcal{S}(t-s)[l\mathcal{R}[u,u](s)]\,ds&=\mathcal{S}(t)f_0+\int_0^t \mathcal{S}(t-s)\mathcal{G}[l](s)\,ds,\label{integrated loss}\\
u(t)+ \int_0^t \mathcal{S}(t-s)[u\mathcal{R}[l,l](s)]\,ds &=\mathcal{S}(t) f_0+\int_0^t \mathcal{S}(t-s)\mathcal{G}[u](s)\,ds.\label{integrated gain}
\end{align}

Given $t\ge 0$, \eqref{integrated loss}-\eqref{integrated gain}, \eqref{gain decomposed}, \eqref{collision equal loss} and \eqref{conslinweight} imply
 \begin{equation*}
 \begin{split}
 \|&u(t)-l(t)\|_{L^1_{xv}}\\
 &\leq \int_0^t \left\|\mathcal{S}(t-s) \big(\mathcal{G}[u]-\mathcal{G}[l] \big)\right\|_{L^1_{xv}}\,ds +\int_0^t\left\|\mathcal{S}(t-s)\big( l\mathcal{R}[u,u]-u\mathcal{R}[l,l] \big) \right\|_{L^1_{xv}}\,ds\\
&\leq\sum_{i=1,2}\int_0^t \left\|\mathcal{G}_i[u,u,u](s)-\mathcal{G}_i[l,l,l](s)\right\|_{L^1_{xv}}\,ds +\int_0^t\left\|\mathcal{L}_i[l,u,u](s)-\mathcal{L}_i[u,l,l](s)\right\|_{L^1_{xv}}\,ds \\
& \lesssim \sum_{\mathcal{T} \in \lbrace \mathcal{L}_1, \mathcal{L}_2, \mathcal{G}_1, \mathcal{G}_2 \rbrace} \sum_{h_1,h_2 \in \lbrace u,l,l-u,u-l \rbrace} \int_0^t \Big (\Vert \mathcal{T}[u-l,h_1,h_2](s) \Vert_{L^1_{xv}}\\
& \hspace{5cm} +\Vert \mathcal{T}[h_1,u-l,h_2](s) \Vert_{L^1_{xv}} +  \Vert \mathcal{T}[h_1,h_2,u-l](s) \Vert_{L^1_{xv}}\Big)  \, ds.
\end{split}
\end{equation*}
Now using \eqref{contraction estimate L1} for $f=u-l$  and \eqref{global bounds} we find that 
\begin{align*}
&    \Vert \mathcal{T}[u-l,h_1,h_2](s) \Vert_{L^1_{xv}} +  \Vert \mathcal{T}[h_1,u-l,h_2](s) \Vert_{L^1_{xv}} +  \Vert \mathcal{T}[h_1,h_2,u-l](s) \Vert_{L^1_{xv}}  \\
\lesssim & \Vert u(s)-l(s) \Vert_{L^1_{xv}}\sum_{h_1,h_2 \in \lbrace u,l,l-u,u-l \rbrace} \Vert \l v \r h_1(s) \Vert_{L^\infty_x L^1_v} \Vert \l v \r^3 h_2(s) \Vert_{L^\infty_{xv}}  \\
 \lesssim & \langle s \rangle^{-3/2}A^2\varepsilon_0^2\Vert u-l \Vert_{L^\infty([0;\infty),L^1_{xv})}  ,
\end{align*}
Hence 
\begin{align*}
     \|u(t)-l(t)\|_{L^1_{xv}} \lesssim A^2\varepsilon_0^2 \Vert u-l \Vert_{L^\infty([0;\infty),L^1_{xv})}\int_0^t\l s\r^{-3/2}\,ds\leq \frac{1}{2} \Vert u-l \Vert_{L^\infty([0;\infty),L^1_{xv})},
\end{align*}
since $0<\varepsilon_0<\bar{\varepsilon}_5$ and $\bar{\varepsilon}_5$ is sufficiently small. Since $t$ was chosen arbitrarily we conclude that $u=l$. Plugging this information into \eqref{integrated loss}, \eqref{integrated gain}, we find that $u,l$ satisfy the same initial value problem as $f,$ thus $u=l=f$ by the uniqueness part of Theorem \ref{thmdispersive}. We conclude that $f\geq 0$.

\end{proof}



\subsection*{Conflict of interest statement}
The authors declare that they have no conflict of interest.

\subsection*{Data availability statement}
This article does not have associated data.


\begin{thebibliography}{1}

\bibitem{ACGM}
R. Alonso, J. Ca\~{n}izo, I. M. Gamba, C. Mouhot, \textit{A new approach to the creation and propagation of exponential moments in the Boltzmann equation}, Comm. Partial Differential Equations \textbf{38} (2013), no. 1, 155–169


\bibitem{AC}
R. Alonso, E. Carneiro, \textit{Estimates for the Boltzmann collision operator via radial symmetry and Fourier transform}, Adv. Math. \textbf{223} (2010), no. 2, 511–528.


\bibitem{ACG} 
R. Alonso, E. Carneiro, I. Gamba, \textit{ Convolution inequalities for the Boltzmann collision operator}, Comm. Math. Phys. \textbf{298} (2010), no. 2, 293–322.

\bibitem{AlGa09} R. Alonso, I.M. Gamba, \textit{Distributional and classical solutions to the Cauchy-Boltzmann problem for
soft potentials with integrable angular cross section}, J. Stat. Phys. \textbf{137} (5–6), 1147–1165 (2009)

\bibitem{Am}
I. Ampatzoglou, \textit{Global well-posedness of the inhomogeneous kinetic wave equation near vacuum}, Kinet. Relat. Models  DOI: 10.3934/krm.2024003 (2024)


\bibitem{AmCoGer} 
I. Ampatzoglou, C. Collot, P. Germain, \textit{Derivation of the kinetic wave equation for quadratic dispersive problems in the inhomogeneous setting}, to appear in Amer. J. Math. (2023)


\bibitem{amgapata22} I. Ampatzoglou, I. M. Gamba, N. Pavlovi\'c, M. Taskovi\'c,
 \textit{Global well-posedness of a binary–ternary Boltzmann equation},
 Ann. Inst. H. Poincaré C Anal. Non Linéaire \textbf{39} (2022), no. 2, 327–369.


\bibitem{AmLeBoltz} 
I. Ampatzoglou, T. L\'eger, \textit{Global dispersive solutions to the Boltzmann equation}, in preparation


\bibitem{AmMiPaTa24} I. Ampatzoglou, J.K. Miller, N. Pavlovi\'c, M. Taskovi\'c, \textit{Inhomogeneous wave kinetic equation and its hierarchy in polynomially weighted  $L^\infty$ spaces}, arXiv:2405.03984 (2024)


\bibitem{Arsenio}
D. Arsenio, \textit{On the global existence of mild solutions to the Boltzmann equation for small data in $L^D$}, Comm. Math. Phys. \textbf{302} (2011), 453-476


\bibitem{BGGL}
C. Bardos, I. M. Gamba, F. Golse, C. D. Levermore, \textit{Global solutions of the Boltzmann equation over $\mathbb{R}^D$  near global Maxwellians with small mass}, Comm. Math. Phys. \textbf{346} (2016), no. 2, 435–467.

\bibitem{beto85} N. Bellomo, G. Toscani, 
\textit{On the Cauchy problem for the nonlinear Boltzmann equation: global existence, uniqueness and asymptotic stability},
J. Math. Phys. \textbf{26} (1985), no. 2, 334–338. 


\bibitem{bo97} A. V. Bobylev,
\textit{Moment inequalities for the Boltzmann equation and applications to spatially homogeneous problems},
J. Statist. Phys. \textbf{88} (1997), no. 5-6, 1183–1214.

\bibitem{BGHS}
T. Buckmaster, P. Germain, Z. Hani, J. Shatah, \textit{Onset of the wave turbulence description of the longtime behavior of the nonlinear Schr\"{o}dinger equation}, Invent. Math. \textbf{225} (2021), no. 3, 787–855.

\bibitem{Chaturvedi}
S. Chaturvedi, \textit{Stability of vacuum for the Boltzmann equation with moderately soft potential}, Ann. PDE \textbf{7} (2021), no. 2, Paper No. 15, 104 pp





\bibitem{CDP-LWP}
T. Chen, R. Denlinger, N. Pavlovi\'c, \textit{
Local well-posedness for Boltzmann's equation and the Boltzmann hierarchy via Wigner transform}, Comm. Math. Phys. \textbf{368} (2019), no. 1, 427–465.


\bibitem{CDP-GWP}
T. Chen, R. Denlinger, N. Pavlovi\'c, \textit{Small data global well-posedness for a Boltzmann equation via bilinear spacetime estimates}, Arch. Ration. Mech. Anal. \textbf{240} (2021), no. 1, 327–381.


\bibitem{CDP-moments}
T. Chen, R. Denlinger, N. Pavlovi\'c, \textit{Moments and regularity for a Boltzmann equation via Wigner transform}, Discrete Contin. Dyn. Syst. \textbf{39} (2019), no. 9, 4979–5015.

\bibitem{CH}
X. Chen, J. Holmer, \textit{Well/ill-posedness bifurcation for the Boltzmann equation with constant collision kernel}, to appear in Ann. PDE


\bibitem{CSZ1}
X. Chen, S. Shen, Z. Zhang, \textit{Well/Ill-posedness of the Boltzmann Equation with Soft Potential}, arXiv:2310.05042


\bibitem{CSZ2}
X. Chen, S. Shen, Z. Zhang, \textit{Sharp Global Well-posedness and Scattering of the Boltzmann Equation}, arXiv:2311.02008

\bibitem{codige24} C. Collot, H. Dietert, P. Germain \textit{Stability and Cascades for the Kolmogorov–Zakharov Spectrum of Wave Turbulence}, Arch. Rational Mech. Anal. Vol. \textbf{248}, no. 7 (2024) 

\bibitem{CG1}
C. Collot, P. Germain, \textit{On the derivation of the homogeneous kinetic wave equation}, to appear in Comm. Pure Appl. Math., arXiv:1912.10368


\bibitem{CG2}
C. Collot, P. Germain, \textit{Derivation of the homogeneous kinetic wave equation: longer time scales}, arXiv:2007.03508


\bibitem{DH1}
Y. Deng, Z. Hani, \textit{On the derivation of the wave kinetic equation for NLS}, Forum Math. Pi \textbf{9} (2021), Paper No. e6, 37 pp.


\bibitem{DH2}
Y. Deng, Z. Hani, \textit{Full derivation of the wave kinetic equation}, Invent. Math. \textbf{233} (2023), no. 2, 543–724.


\bibitem{DH3}
Y. Deng, Z. Hani, \textit{Propagation of chaos and the higher order statistics in the wave kinetic theory}, arXiv:2110.04565


\bibitem{DH4}
Y. Deng, Z. Hani, \textit{Long time justification of wave turbulence theory}, arXiv:2311.10082


\bibitem{DIP}
Y. Deng, A. Ionescu, F. Pusateri, \textit{On the wave turbulence theory of 2D gravity waves, I: deterministic energy estimates}, arXiv:2211.10826

\bibitem{de93} L. Desvillettes, 
\textit{Some applications of the method of moments for the homogeneous Boltzmann and Kac equations}, 
Arch. Ration. Mech. Anal., \textbf{123} (1993) pp. 387–404.

\bibitem{DV} L. Desvillettes, C. Villani,
\textit{On the trend to global equilibrium for spatially inhomogeneous kinetic systems}, Invent. Math. \textbf{159}, (2005) 245-316


\bibitem{EsMe24} M. Escobedo, A. Menegaki, \textit{Instability of singular equilibria of a wave kinetic equation}, 	arXiv:2406.05280


\bibitem{EV}
M. Escobedo, J.J.L. Vel\'{a}zquez, \textit{On the theory of weak turbulence for the nonlinear Schr\"{o}dinger equation}, Mem. Amer. Math. Soc. \textbf{238} (2015), no. 1124, v+107 pp.


\bibitem{GeIoTr}
P. Germain, A. D. Ionescu, M.-B. Tran, \textit{Optimal local well-posedness theory for the kinetic wave equation}, J. Funct. Anal. \textbf{279} (2020), no. \textbf{4}, 108570, 28 pp.


\bibitem{HaShZh}
Z. Hani, J. Shatah and H. Zhu, \textit{Inhomogeneous Turbulence for Wick NLS}, to appear in Comm. Pure Appl. Math.


\bibitem{ha62} K. Hasselmann, \textit{On the non-linear energy transfer in a gravity-wave spectrum. I. General theory}, J. Fluid Mech. \textbf{12} (1962), 481–500.


\bibitem{ha63} K. Hasselmann, \textit{On the non-linear energy transfer in a gravity wave spectrum. II. Conservation theorems; wave-particle analogy; irreversibility}, J. Fluid Mech. \textbf{15} (1963), 273–281. 


\bibitem{HJKL} L.-B. He, J.-C. Jiang, H.-W. Kuo, M.-H. Liang 
\textit{The $L^p$ estimate for the gain term of the Boltzmann collision operator and its application}, arXiv preprint arXiv:2404.05517


\bibitem{illner-shinbrot} R. Illner, M. Shinbrot, 
\textit{The Boltzmann Equation: Global Existence for a Rare Gas in an Infinite Vacuum},
Commun. Math. Phys. \textbf{95}, 217-226 (1984).





\bibitem{KS}
S. Kaniel, M. Shinbrot, \textit{The Boltzmann Equation: Uniqueness and Local Existence}, Comm. Math. Phys. \textbf{58}, 65-84 (1978)





\bibitem{Ma}
X. Ma, \textit{Almost sharp wave kinetic theory of multidimensional KdV type equations with $d \geq 3$}, arXiv:2204.06148 (2022)

\bibitem{me23} A. Menegaki, 
\text{$L^2$-stability near equilibrium for the 4 waves kinetic equation}, Kinet. Relat. Models, (2023). DOI: 10.3934/krm.2023031.


\bibitem{miwe99} S. Mischler, B. Wennberg,
\textit{On the spatially homogeneous Boltzmann equation},
 Ann. Inst. H. Poincar\' e C Anal. Non Linéaire \textbf{16} (1999), no. 4, 467–501.

\bibitem{Na}
S. Nazarenko, \textit{Wave turbulence}, Lecture Notes in Physics, \textbf{825}. Springer, Heidelberg, 2011


\bibitem{Pe}
R. Peierls, \textit{Zur kinetischen theorie der W\"{a}rmeleitung in kristallen}, Annalen der Physik \textbf{395} (1929) 1055-1101


\bibitem{ST}
G. Staffilani, M.-B. Tran, \textit{On the wave turbulence theory for a stochastic KdV type equation}, arXiv:2106.09819 (2021)


\bibitem{to86}  G. Toscani, 
\textit{On the nonlinear Boltzmann equation in unbounded domains},
Arch. Rational Mech. Anal. \textbf{95} (1986), no. 1, 37–49.


\bibitem{Villani} 
C. Villani, \textit{A review of mathematical topics in collisional kinetic theory, Handbook of mathematical
fluid dynamics}, Vol. I, 71-305, North-Holland, Amsterdam, 2002.

\bibitem{ZLF}
V. E. Zakharov, V. S. L’vov, G. Falkovich, \textit{Kolmogorov spectra of turbulence I: Wave turbulence}, Springer Science \& Business Media, 1992






\end{thebibliography}
\end{document}